\documentclass[]{siamart1116}
\usepackage{mathrsfs,latexsym,amsmath,amssymb,graphicx,textcomp,multirow,stmaryrd,caption}
\usepackage{enumerate}
\usepackage{epstopdf}
\newcommand{\vertiii}[1]{{\left\vert\kern-0.25ex\left\vert\kern-0.25ex\left\vert #1
    \right\vert\kern-0.25ex\right\vert\kern-0.25ex\right\vert}}
\newsiamthm{remark}{Remark}

\newcommand{\bx}{{\bf x}}

\begin{document}
\title{A Multilevel Monte Carlo Ensemble Scheme for Solving Random Parabolic PDEs}
\author{
Yan Luo \thanks{School of Mathematical Sciences, University of Electronic Science and Technology of China, No.2006, Xiyuan Ave, West Hi-Tech Zone, Chengdu, Sichuan 611731, China and School of Mathematics, Sichuan University, No.24 South Section 1, Yihuan Road, Chengdu, Sichuan 610064, China. Research supported by the Young Scientists Fund of the National Natural Science Foundation of China grant 11501088.}
\and Zhu Wang\thanks{Department of Mathematics,
University of South Carolina, 1523 Greene Street, Columbia, SC 29208, USA (\email{wangzhu@math.sc.edu}). Research  supported by the
U.S. National Science Foundation grant DMS-1522672 and the U.S. Department of Energy grant DE-SC0016540.}
}
\maketitle

\begin{abstract}
A first-order, Monte Carlo ensemble method has been recently introduced for solving parabolic equations with random coefficients in \cite{luo2017ensemble}, which is a natural synthesis of the ensemble-based, Monte Carlo sampling algorithm and the ensemble-based, first-order time stepping scheme.
With the introduction of an ensemble average of the diffusion function, this algorithm
leads to a single discrete system with multiple right-hand sides
for a group of realizations, which could be solved more efficiently than
a sequence of linear systems.
In this paper, we pursue in the same direction and develop a new multilevel Monte Carlo ensemble method for solving random parabolic partial differential equations.
Comparing with the approach in \cite{luo2017ensemble}, this method possesses a high-order accuracy in time and further reduces the computational cost by using the multilevel Monte Carlo method.
Rigorous numerical analysis shows the method achieves the optimal rate of convergence.
Several numerical experiments are presented to illustrate the theoretical results.
%The method can also be applied to solve a group of deterministic PDEs in which members are independent and are subject to different diffusion functions, initial conditions, boundary conditions, and body forces.
\end{abstract}


%==========================================================
\section{Introduction}
%==========================================================
In this paper, we consider numerical solutions to the following unsteady
heat conduction equation in a random, spatially varying medium: to find a random function, $u:
\Omega\times\overline{D}\times [0,T]\rightarrow \mathbb{R}$
satisfying almost surely (a.s.)
\begin{equation}
\left\{
\begin{aligned}
&u_t(\omega, \mathbf{x}, t)-\nabla\cdot[(a(\omega,\mathbf{x})\nabla u(\omega, \mathbf{x}, t)]=f(\omega, \mathbf{x},  t), &\text{ in } \Omega\times D\times[0, T] \\
 &u(\omega,\mathbf{x}, t)=g(\omega,{\bf x},t), &\text{ on } \Omega\times \partial D\times [0, T] \\
 &u(\omega, \mathbf{x}, 0) = u^0(\omega,{\mathbf{x}}), &\text{ in }\Omega \times D
 \end{aligned}
 \right.
 ,
 \label{eq:rand}
 \end{equation}
where $D$ is a bounded Lipschitz domain in $\mathbb{R}^d$ and $(\Omega, \mathcal{F}, P)$ is a probability space with the sample space $\Omega$, $\sigma$-algebra $\mathcal{F}$, and probability measure $P$; 
diffusion coefficient $a: \Omega\times D\rightarrow \mathbb{R}$ and body force $f: \Omega\times D\times [0, T] \rightarrow \mathbb{R}$ are random fields with continuous and bounded covariance functions.
%It is required that $a(\omega,\mathbf{x})$ is uniformly coercive. That is, there exists a positive constant $\theta$ such that

Many numerical methods, either intrusive or non-intrusive, have been developed for random partial differential equations (PDEs), see, e.g., in the review papers \cite{gunzburger2014stochastic,xiu2009fast} and the references therein. 
For the random steady or unsteady heat equation, non-intrusive numerical methods such as Monte Carlo methods are known for easy implementation but requiring a very large number of PDE solutions to achieve small errors; while intrusive methods such as the stochastic Galerkin or collocation approaches can achieve faster convergence but would require the solution of discrete systems that couple all spatial and probabilistic degrees of freedom \cite{babuvska2007stochastic,babuska2004galerkin,xiu2003new}. 
To improve the computational efficiency of the non-intrusive approaches, other sampling methods such as quasi-Monte Carlo, multilevel Monte Carlo (MLMC), Latin hypercube sampling and Centroidal Voronoi tessellations can be used \cite{niederreiter1992random,helton2003latin,du1999centroidal,romero2006comparison}.
In particular, the MLMC method is designed to greatly reduce the computational cost by performing most simulations at a low accuracy, while running relatively few simulations at a high accuracy. It was first introduced by Heinrich \cite{heinrich2001multilevel} for the computation of high-dimensional, parameter-dependent integrals and was analyzed extensively by Giles \cite{giles2008multilevel,giles2008improved} in the context of stochastic differential equations in mathematical finance. In \cite{cliffe2011multilevel}, Cliffe et al. applied the MLMC method to the elliptic PDEs with random coefficients and demonstrated its numerical superiority. Under the assumptions of uniform coercivity and boundedness of the random parameter, numerical error of the MLMC approximation has been analyzed in \cite{barth2011multi}. The result was extended in \cite{charrier2013finite} for random elliptic problems with weaker assumptions on the random parameter and a limited spatial regularity.

Overall, the above mentioned sampling methods are ensemble-based. 
To quantify probabilistic uncertainties in a system governed by random PDEs, an ensemble of independent realizations of the random parameters needs to be considered. 
In practice, this process would involve solving a group of deterministic PDEs corresponding to all the realizations.
A straightforward solution strategy is to find numerical approximate solutions of the deterministic PDEs from a sequence of discrete linear systems.
Obviously, this approach ignores any possible relationships among the group members, thus cannot improve the overall computational efficiency.
To speed up the group of simulations, current active research mainly starts from the perspective of numerical linear algebra, and develops iterative algorithms that can take advantage of the relationship in the sequence of discrete systems.
For instance, subspace recycling techniques such as GCRO with deflated restarting have been introduced in \cite{parks2006recycling} for accelerating the solutions of slowly-changing linear systems, which is further developed in \cite{ahuja2010krylov} for climate modeling and uncertainty quantification applications.
For sequences sharing a common coefficient matrix, block iterative algorithms \cite{gutknecht2006block,meng2014block,o1980block,o1987parallel,simoncini1996convergence} have been developed to solve the system with many right-hand sides. The algorithms have been used to accelerate convergence even when there is only one right-hand side in \cite{chronopoulos2010block,o1987parallel}.
The block version of GCRO with deflated restarting was introduced in \cite{parks2016block}, and its high-performance implementation is available in the $\mathsf{Belos}$ package of the Trilinos project developed at US Sandia National Laboratories.

Recently, the Monte Carlo ensemble method was introduced by the authors of this paper for solving the random heat equations in \cite{luo2017ensemble}. This method is motivated by the ensemble-based time stepping algorithm, which was proposed for solving Navier-Stokes incompressible flow ensembles in \cite{jiang2014algorithm,jiang2015higher,jiang2015analysis,jiang2015numerical,takhirov2015time,jiang2017second} and for simulating ensembles of parameterized Navier-Stokes flow problems in \cite{gunzburger2016efficient,gunzburger2017second}. It has been extended to MHD flows in \cite{mohebujjaman2017efficient} and to low-dimensional surrogate models in \cite{gunzburger2017ensemble,Gunzburger2016higher}.
The main idea is to manipulate the numerical scheme so that all the simulations in the ensemble could share a common coefficient matrix.
As a consequence, simulating the ensemble only requires to solve a single linear system with multiple right-hand sides, which could be easily handled by a block iterative solver and, thus, improves the overall computational efficiency.
Thus, the Monte Carlo ensemble method was proposed in \cite{luo2017ensemble} for synthesizing a first-order, ensemble-based time-stepping and the ensemble-based, Monte Carlo sampling method in a natural way, which speeds up the numerical approximation of the random parabolic PDE solutions and other possible quantities of interest.
%Monte Carlo (MC) methods are widely used in statistical simulations for the estimation of expectations
However, it is known that the Monte Carlo method, although easy for implementations, is a computationally expensive random sampling approach. Therefore, in this paper, we develop a new method for solving the same random heat equations with a better accuracy and efficiency: the new method is second-order accurate in time, which improves the temporal accuracy of our previous work; it employs the idea of multilevel Monte Carlo methods, which improves the computational efficiency comparing with the Monte Carlo. We further perform theoretical analysis on the method and present numerical tests that illustrate our theoretical findings. 
Upon the completion of this paper, we found the second-order ensemble-based time-stepping scheme had been discussed in the preprint \cite{fiordilino2017ensemble}, however, without using the efficient sampling method in uncertainty quantification. 
%Furthermore, the stability constant we provide here is better than the one provided in \cite{fiordilino2017ensemble}.

The rest of this paper is organized as follows.
 In Section \ref{sec:notation}, we present some notation and mathematical preliminaries.
 In Section \ref{sec:algorithm}, we introduce the multilevel Monte Carlo ensemble scheme in the context of finite element (FE) methods.
In Section \ref{sec:analysis}, we analyze the proposed algorithm and prove its stability and convergence.
Numerical experiments are presented in Section \ref{sec:num}, which illustrate the effectiveness of the proposed scheme on random parabolic problems.
A few concluding remarks are given in Section \ref{sec:con}.

\section{Notation and preliminaries}\label{sec:notation}

Denote the $L^2(D)$ norm and inner product by $\|\cdot\|$ and $(\cdot,\cdot)$, respectively.
 Let $W^{s, q}(D)$ be the Sobolev space of functions having generalized
 derivatives up to the order $s$ in the space $L^q(D)$, where $s$ is a nonnegative integer and $1\leq q\leq +\infty$.
 The equipped Sobolev norm of $v\in W^{s,q}(D)$ is denoted by $\|v\|_{W^{s,q}(D)}$.
 When $q=2$, we use the notation $H^s(D)$ instead of $W^{s, 2}(D)$.
 As usual, the function space $H_0^1(D)$
 is the subspace of $H^1(D)$ consisting  of functions that vanish
 on the boundary of $D$ in the sense of trace, equipped with the
 norm $\|v\|_{H^1_0(D)}=\left(\int_D|\nabla v|^2\, d\bx\right)^{1/2}$.
 When $s= 0$, we shall keep the notation with $L^q(D)$ instead of $W^{0, q}(D)$.
 The space $H^{-s}(D)$ is the dual space of bounded linear
 functions on $H_0^s(D)$. A norm for $H^{-1}(D)$ is defined by
 $ \|f\|_{-1}=\sup\limits_{0\neq v\in H_0^1(D)}\frac{(f,v)}{\|\nabla
 v\|}$.
 
 Let $(\Omega,\mathcal{F},P)$ be a complete probability space. If $Y$ is a random variable in the space and belongs to $L_P^1(\Omega)$, its expected value is defined by
 $$
 \mathbb{E}[Y]=\int_{\Omega}Y(\omega)dP(\omega).
 $$
 \par
 With the multi-index notation, $\alpha=(\alpha_1,\ldots,\alpha_d)$ is a $d$-tuple of nonnegative
 integers with the length $\alpha$ is given by $|\alpha|=\sum_{i=1}^d\alpha_i$. The stochastic Sobolev spaces
 $\widetilde{W}^{s,q}(D)=L^q_P(\Omega,W^{s,q}(D))$ containing
 stochastic functions, $v:\Omega\times D\rightarrow R$, that are
 measurable with respect to the product $\sigma$-algebra $\mathcal{F}\bigotimes
 B(D)$ and equipped with the averaged norms $\|v\|_{\widetilde{W}^{s,q}(D)}=(E[\|v\|^q_{W^{s,q}(D)}])^{1/q}
 =(E[\sum_{|\alpha|\leq s}\int_D|\partial^\alpha v|^qd\bx])^{1/q},1\leq
 q<+\infty$. Observe that if $v\in \widetilde{W}^{s,q}(D)$, then $v(\omega,\cdot)\in
 W^{s,q}(D)$ a.s. and $\partial^\alpha v(\cdot, \bx)\in
 L^q_P(\Omega)$ a.e. on $D$ for $\forall |\alpha|\leq s$. 
 In particular, $\widetilde{W}^{s,2}(D)$ is denoted by $\widetilde{H}^s(D)\simeq L^2_P(\Omega)\bigotimes  H^s(D)$. In this paper, we consider the tensor product Hilbert space $H=\widetilde{L}^2(H^1_0(D); 0,T)\simeq L^2_P(\Omega; H_0^1(D); 0,T)$
 endowed with the inner product $(v,u)_H\equiv E\left[\int_0^T\int_D\nabla v\cdot\nabla u\,d\bx\,dt\right]$.

\section{Ensemble-based multilevel Monte Carlo method}\label{sec:algorithm}
Given statistical information on the inputs of a random/stochastic PDE, uncertainty quantification implements the task of determining statistical information about an output of interest that depends on the PDE solutions. 
When stochastic sampling methods such as the Monte Carlo are used to solve \eqref{eq:rand}, one has to find approximate solutions to an ensemble of independent realizations, that is, deterministic PDEs at randomly selected sample values. 
Usually, each numerical simulation is implemented separately, thus the total computational cost is simply multiplied as the sampling set becomes larger. 
To improve the efficiency, we propose an ensemble-based multilevel Monte Carlo method in this paper, which is an extension of the Monte Carlo ensemble method we introduced in \cite{luo2017ensemble}. 
The new approach outperforms the previous one in both accuracy and efficiency, which is due to the combination of a second-order, ensemble-based time stepping scheme and the multilevel Monte Carlo method.

Next, we present the algorithm in the context of numerical solutions to random PDEs \eqref{eq:rand}.  
For the spatial discretization, we use conforming finite elements, although other numerical methods could be applied as well. 
To fit in the hierarchic nature of multilevel Monte Carlo methods, we consider a sequence of quasi-uniform meshes comprising a set of k-shape regular triangles (or tetrahedra), $\{\mathcal{T}_l\}_{l=0}^L$, for a polygonal (or polyhedral) domain $D$. 
Denote the mesh size of $\mathcal{T}_{l}$ by
\begin{equation*}
h_l = \max\limits_{K\in \mathcal{T}_l}{\textrm{diam } K}.
% =: \max \limits_{K\in \mathcal{T}_l}{h_K}.
\end{equation*}
% We recall (see, e.g. [8,9]) that the nested family $\{\mathcal{T}_l\}_{l=0}^{\infty}$ of
%regular, simplicial meshes is called $\kappa$-shape regular if and
%only if there exists a $\kappa<\infty$ such that $\kappa:=
%\sup\limits_{l}\kappa_l= \sup\limits_{l}\max\limits_{K\in
%\mathcal{T}_l}\frac{h_K}{\rho_K}$. Here $\rho_K$ is the radius of
%the largest ball that can be inscribed into any $K\in\mathcal{T}_l$.
Assume the sequence is generated by uniform mesh refinements satisfying 
\begin{equation}
\label{eq:meshsize}
h_l=2^{-l}h_0. 
\end{equation}
%The nested family $\{\mathcal{T}_l\}_{l=0}^{\infty}$ of regular, simplicial
%triangulations obtained in this way is $\kappa$-shape regular,
%since $\kappa_l=\kappa_0=\kappa$. 
Define the function space 
$H_g^1(D) = \{v\in H^1(D): v|_{\partial D}= g\}$
and the FE space 
%\mathcal{S}^m(D,\mathcal{T}_l)
\begin{equation*}
V_l^g:=\{v\in H_g^1(D) \cap H^{m+1}(D): v|_K
\text{ is a polynomial of degree}\,\,  m  \text{ for }\forall
K\in \mathcal{T}_l\}
\end{equation*}
%The family of FE spaces that we employ is
%$V_l=\{\mathcal{S}^m(D,\mathcal{T}_l)\}^\infty_{l=0}$, 
for non-negative integer $m$. 
%which consists of continuous, piecewise functions on $\mathcal{T}_l$ that satisfies the homogeneous Dirichlet boundary condition. 
%
%We will discuss the choice of $\Delta t_l$ in Section ?2 and Section ?3.\par
The sequence of finite element spaces satisfies 
$$
V_0^g\subset V_1^g \subset \cdots \subset
V_l^g \subset \cdots \subset V_L^g.
$$
Denoted by $u_l(\omega, \bx, t_n)$ the finite element solution in $V_l^g$ at the time instance $t_n$. 
The MLMC FE solution at the $L$-th level mesh can be written as 
$$
u_L(\omega, \bx, t_n)=\sum\limits_{l=1}^{L}\big( u_l(\omega, \bx, t_n) - u_{l-1}(\omega, \bx, t_n) \big)+u_0(\omega, \bx, t_n).
$$
By linearity of the expectation operator $\mathbb{E}[\cdot]$, we have 
\begin{eqnarray*}
\mathbb{E}\big[ u_L(\omega, \bx, t_n) \big]&=&\mathbb{E}\Big[\sum\limits_{l=1}^L\big(u_l(\omega, \bx, t_n) - u_{l-1}(\omega, \bx, t_n) \big)+u_0(\omega, \bx, t_n)\Big]\\
&=&\sum\limits_{l=1}^L\mathbb{E}\big[ u_l(\omega, \bx, t_n) - u_{l-1}(\omega, \bx, t_n) \big]+\mathbb{E}\big[ u_0(\omega, \bx, t_n) \big].
%&=&\sum\limits_{l=1}^L\Big( \mathbb{E}\big[ u_l(t_n) \big]-\mathbb{E}\big[ u_{l-1}(t_n) \big] \Big)+\mathbb{E}\big[ u_0(t_n) \big],
\end{eqnarray*}
Numerically, the expected value of the FE solution on the $l$-th level, $\mathbb{E}[u_l(\omega,\bx, t_n)]$ is approximated by the sampling average  $\Psi_{J_l}^n = \Psi_{J_l}[u_l(\omega,\bx,t_n)]=\frac{1}{J_l}\sum_{j=1}^{J_l}u_{l}(\omega_j,\bx,t_n)$, where $J_l$ is the number of selected samples. Correspondingly, $\mathbb{E}[u_L(\omega, \bx, t_n)]$ is approximated by   
\begin{equation}
\Psi[u_L(\omega, \bx, t_n)]:= \sum\limits_{l=1}^L\big(\Psi_{J_l}[u_l(\omega, \bx, t_n)-u_{l-1}(\omega, \bx, t_n)]\big)
+\Psi_{J_0}[u_0(\omega, \bx, t_n)].
\label{E_L}
\end{equation}

It is seen that, at each mesh level, a group of simulations needs to be implemented.   
In order to improve the computational efficiency, we introduce the following 
ensemble-based multilevel Monte Carlo (EMLMC) method.

For simplicity of presentation, we assume that, at the $l$-th level, a uniform time partition on $[0, T]$ with the time step $\Delta t_l$ is used for the simulations  and further set $N_l= T/\Delta t_l$; a set of $J_l$ samples are taken that are independent, identically distributed (i.i.d.), and functions at random samples $\{\omega_j\}_{j=1}^{J_l}$ are denoted by $a_j\equiv a(\omega_j,\cdot)$, $f_j\equiv f(\omega_j,\cdot, \cdot)$, $g_j\equiv g(\omega_j,\cdot, \cdot)$, and $u_j^0\equiv u^0(\omega_j,\cdot)$, and define the ensemble mean of the diffusion coefficient functions by
$$\overline{a}_l:=\frac{1}{J_l}\sum\limits_{j=1}^{J_l} a(\omega_j,{\mathbf{x}}).$$
Here, we note that the corresponding exact solutions $\{u(\omega_j,{\bf{x}}, t)\}_{j=1}^{J_l}$ are i.i.d.
Let $u_{j,l}^n=u_{l}(\omega_j,\mathbf{x},t_n)$, the finite element approximation of $u(\omega_j,{\bf{x}}, t_n)$ at the $l$-th level.

The {\em ensemble-based multilevel Monte Carlo method (EMLMC)} applied to \eqref{eq:rand} solves the following group of simulations at the $l$-th level: for $j=1, \ldots, J_l$, 
given $u_{j,l}^{0}$ and $u_{j,l}^1$, to find $u_{j,l}^{n+1}\in V_l^g$ such that, 
\begin{equation}
\begin{aligned}
&\left(\frac{3u_{j,l}^{n+1}-4u_{j,l}^n+u_{j,l}^{n-1}}{2\Delta t_l},v_l\right)
+(\overline{a}_l\nabla u_{j,l}^{n+1},\nabla v_l)  \\
&\hspace{1cm}=- \big( (a_j-\overline{a}_l)\nabla(2 u_{j,l}^n-u_{j,l}^{n-1}),\nabla v_l \big) + (f_{j}^{n+1},v_l),  \quad \forall \,v_l\in V_l^0,
\end{aligned}
\label{eqn:ens_rand}
\end{equation}
for $n=1,\ldots, N_l-1$. 
Once the numerical solutions at all the $L$ levels are found, the EMLMC approximates the SPDEs solution at the time instance $t_n$, $\mathbb{E}[u(t_n)]$, by \eqref{E_L}. 
%$$\Psi[u_L^n] = \frac{1}{J_0}\sum_{j=1}^{J_0}u_{j,0}^n+\sum_{l=1}^{L} 
%\Big[ \frac{1}{J_l}\sum_{j=1}^{J_l} \big( u_{j,l}^n-u_{j,l-1}^n \big) \Big].$$ 
Meanwhile, given a quantity of interest $Q(u)$, one can analyze the outputs from the ensemble simulations, $Q(u_h(\omega_1, \cdot, \cdot))$, $\ldots,  Q(u_h(\omega_J, \cdot, \cdot))$, for extracting the underlying stochastic information of the system.

It is seen that the EMLMC naturally combines the ensemble-based sampling method and the ensemble-based time stepping algorithms, and inherits advantages from both sides. As the MLMC, the method can reduce the computational cost by balancing the time step size, mesh size, and the number of samples at each level. Since the coefficient matrix of the discrete linear system \eqref{eqn:ens_rand} is independent of $j$,  for evaluating $J_l$ realizations, one only needs to solve one linear system with multiple right-hand sides. This leads to great computational savings: when the number of degrees of freedom is small, one can perform the LU factorization once instead of $J_l$ times; when the number of degrees of freedom is large, one can use the block iterative algorithms to find the ensemble solution more efficiently than solving a sequence of simulations. 
Next, we will analyze the stability and asymptotic error estimate of the EMLMC method.  

\section{Stability and error estimate}\label{sec:analysis}
To simplify the presentation, we only consider equation  \eqref{eq:rand} with the homogeneous boundary condition (that is, $g=0$ and $u_{j,l}^{n+1}\in V_l^0$ in the FE weak form \eqref{eqn:ens_rand}), while the nonhomogeneous cases can be similarly analyzed by incorporating the method of shifting. 
Meanwhile, we will include numerical test cases with nonhomogeneous boundary conditions in Section 5.
As the EMLMC approximation is based on the MC solutions at various levels, we first analyze the ensemble-based single-level Monte Carlo in Subsection \ref{sec:analysis_1l} and derive the error estimate for EMLMC in Subsection \ref{sec:analysis_ml}.

Assume the exact solution of \eqref{eq:rand} is smooth enough, in particular, 
\begin{equation*}
u_j\in \widetilde{L}^2(H_0^1(D) \cap H^{m+1}(D); 0,T)
\cap \widetilde{H}^1(H^{m+1}(D); 0,T) \cap \widetilde{H}^2(L^2(D); 0,T)
\end{equation*}
and suppose
$$f_j\in {\widetilde{L}}^2 \left({H}^{-1}(D); 0,T\right).$$
Assume the following two conditions hold: 
\begin{enumerate}
\item[$\mathsf{(i)}$] There exists a positive constant $\theta$ such that
\begin{equation*}
P\{\omega\in\Omega;
\min\limits_{\mathbf{x}\in\overline{D}}a(\omega,\mathbf{x})>\theta\}=1.
\end{equation*}
\item[$\mathsf{(ii)}$] There exists a positive constant $\theta_+$, for $l=0, \ldots, L$, such that
\begin{equation*}
P\{\omega_j\in\Omega; |a(\omega_j,{\bf x})-\overline{a}_l|_{\infty}\leq\theta_+\}=1.
\end{equation*}
\end{enumerate}
Here, condition $\mathsf{(i)}$ guarantees the uniform coercivity a.s. and condition $\mathsf{(ii)}$ gives an upper bound of the distance from coefficient $a(\omega_j,{\bf x})$ to the ensemble average $\overline{a}_l$ a.s.

\subsection{Ensemble-based single-level Monte Carlo finite element method}\label{sec:analysis_1l}

When $\mathbb{E}[u(t_n)]$ is numerically approximated by $\Psi_{J_l}^n$, the associated approximation error can be separated into two parts:
\begin{eqnarray*}
\mathbb{E}[u(t_n)]-\Psi_{J_l}^n
=\left(\mathbb{E}[u_j(t_n)]-\mathbb{E}[u_{j,l}^n]\right)
+\left(\mathbb{E}[u_{j,l}^n]-\Psi_{J_l}^n\right)
:=\mathcal{E}_l^n+\mathcal{E}_S^n, 
\end{eqnarray*}
where we use the fact that $\mathbb{E}[u(t_n)]=\mathbb{E}[u_j(t_n)]$. 
The finite element discretization error, $\mathcal{E}_l^n=\mathbb{E}[u_j(t_n)-u_{j,l}^n]$, is controlled by the size of spatial triangulations $\mathcal{T}_l$ and time step; while the statistical error, $\mathcal{E}_S^n=\mathbb{E}[u_{j,l}^n]-\Psi_{J_l}^n$, is dominated by the number of realizations. 
Next, we will first discuss the stability of the ensemble scheme (\ref{eqn:ens_rand}) at the $l$-th level (Theorem \ref{th:stability_R}), derive the bounds for $\mathcal{E}_S^n$ (Theorem \ref{th:E_S}) and $\mathcal{E}_l^n$ (Theorem \ref{th:error}), and then obtain the asymptotic error estimation (Theorem \ref{th:Psi}).

%Note that, once $\omega_j$ is fixed, the scheme
%(\ref{eqn:ens_rand}) is a second-order, deterministic parabolic
%equation with a diffusion coefficient independent of time.
%An application of the technique which used in deterministic PDEs
%yields the following stability for the finite element solutions
%$u_{l}(\omega_j,\cdot, \cdot)$.
\begin{theorem} \label{th:stability_R}
Under conditions $\mathsf{(i)}$ and $\mathsf{(ii)}$, the scheme (\ref{eqn:ens_rand}) is stable provided that
 \begin{equation}
 \theta > 3\theta_+.
 \label{con1}
 \end{equation}
Furthermore, the numerical solution to (\ref{eqn:ens_rand}) satisfies
\begin{equation}
\begin{aligned}
&\frac{1}{4}\mathbb{E}\big[\|u_{j,l}^{N_l}\|^2\big]+\frac{1}{4}\mathbb{E}\big[\|2u_{j,l}^{N_l}-u_{j,l}^{N_l-1}\|^2\big]
+\frac{\theta}{2}\Delta t_l\mathbb{E}\big[\|\nabla
u_{j,l}^{N_l}\|^2] \\
%+\frac{\theta}{6}\Delta t_l\mathbb{E}\big[\|\nabla u_{j,l}^{N_l-1}\|^2]\\
%+\frac{1}{4}\sum\limits_{n=1}^{N_l-1}\mathbb{E}\big[\|u_{j,l}^{n+1}-2u_{j,l}^n+u_{j,l}^{n-1}\|^2\big]\\
&\qquad+\Big(\frac{\theta}{3}-\theta_+\Big)\Delta t_l\sum\limits_{n=1}^{N_l-1}\mathbb{E}\big[\|\nabla
u_{j,l}^n\|^2\big] \\
%+\Big(\frac{\theta}{6}-\frac{\theta_+}{2}\Big)\Delta t_l\sum\limits_{n=1}^{N_l-1}\mathbb{E}\big[\|\nabla
%u_{j,l}^{n-1}\|^2\big] \\
&\leq 
\frac{\Delta t_l}{2(\theta-3\theta_+)}\sum\limits_{n=1}^{N_l-1}\mathbb{E}\big[\|f_j^{n+1}\|_{-1}^2\big]
+\frac{1}{4}\mathbb{E}\big[\|u_{j,l}^1\|^2\big]+\frac{1}{4}\mathbb{E}\big[\|2u_{j,l}^1-u_{j,l}^0\|^2\big]\\
 & +\frac{\theta}{2}\Delta t_l\mathbb{E}\big[\|\nabla u_{j,l}^1\|^2\big]+\frac{\theta}{6}\Delta t_l\mathbb{E}\big[\|\nabla u_{j,l}^0\|^2\big].
 \end{aligned}
\label{th:eq_sta}
\end{equation}
%where $C$ is a generic positive constant independent of $h$ and $\Delta t_l$.
 \end{theorem}
\begin{proof}
Choosing $v_h=u_{j,l}^{n+1}$ in (\ref{eqn:ens_rand}), we obtain \\
\begin{equation}
\begin{aligned}
&\left(\frac{3u_{j,l}^{n+1}-4u_{j,l}^n+u_{j,l}^{n-1}}{2\Delta t_l},u_{j,l}^{n+1}\right) 
+\left(\overline{a}_l\nabla u_{j,l}^{n+1},\nabla u_{j,l}^{n+1} \right)\\
&\qquad=-\left((a_j-\overline{a}_l)\nabla (2u_{j,l}^n-u_{j,l}^{n-1}),\nabla u_{j,l}^{n+1} \right) 
+ \left(f_{j}^{n+1}, u_{j,l}^{n+1}\right).
\end{aligned}
\end{equation}
Multiplying both sides by $\Delta t_l$, integrating over the probability space and considering the coercivity, we get
\begin{equation}
\begin{aligned}
&\frac{1}{4}\mathbb{E}\big[\|u_{j,l}^{n+1}\|^2
+\|2u_{j,l}^{n+1}-u_{j,l}^n\|^2\big]
-\frac{1}{4}\mathbb{E}\big[\|u_{j,l}^n\|^2+\|2u_{j,l}^n-u_{j,l}^{n-1}\|^2\big]\\
& + \frac{1}{4}\mathbb{E}\big[\|u_{j,l}^{n+1}-2u_{j,l}^n+u_{j,l}^{n-1}\|^2\big]
+\Delta t_l \theta\,\mathbb{E}\big[ \|\nabla u_{j,l}^{n+1}\|^2\big] \\
&\leq \Delta t_l \mathbb{E}\big[ \big|\big(f_j^{n+1},u_{j,l}^{n+1}\big)\big| \big]
+\Delta t_l \theta_+ \mathbb{E} \big[ \big|\big(\nabla(2u_{j,l}^n-u_{j,l}^{n-1}),\nabla u_{j,l}^{n+1}\big)\big| \big]. 
\label{S1}
\end{aligned}
\end{equation}
Apply Young's inequality to the terms on the right-hand side (RHS), we have, for any
$\beta_i>0, i=1,2,3$,
\begin{equation}
\mathbb{E}\big[ \big|(f_j^{n+1},u_{j,l}^{n+1})\big| \big]
\leq 
\frac{\beta_1}{4}\mathbb{E}\big[ \|\nabla u_{j,l}^{n+1}\|^2 \big]
+\frac{1}{\beta_1}\mathbb{E}\big[ \|f_j^{n+1}\|_{-1}^2 \big], 
\label{e00}
\end{equation}
and
\begin{equation}
\begin{aligned}
&\mathbb{E}\left[ \big| \big(\nabla(2u_{j,l}^n-u_{j,l}^{n-1}),\nabla u_{j,l}^{n+1}\big) \big| \right]
=\mathbb{E}\left[  \big|(2\nabla u_{j,l}^n,\nabla u_{j,l}^{n+1})-(\nabla u_{j,l}^{n-1},\nabla u_{j,l}^{n+1})\big| \right]\\
%&\quad\leq \Delta t_l\theta_+\big(\frac{2}{\beta_2}\mathbb{E}\big[\|\nabla
%u_{j,l}^n\|^2\big]+\frac{\beta_2}{2}\mathbb{E}\big[\|\nabla
%u_{j,l}^{n+1}\|^2\big]\big)+\Delta t_l\theta_+\big(\frac{1}{2\beta_3}\mathbb{E}\big[\|\nabla
%u_{j,l}^{n-1}\|^2\big]\nonumber\\
%&\quad\quad+\frac{\beta_3}{2}\mathbb{E}\big[\|\nabla u_{j,l}^{n+1}\|^2\big]\big)\nonumber\\
&\leq \frac{\beta_2+\beta_3}{2}\mathbb{E}\big[\|\nabla u_{j,l}^{n+1}\|^2\big]
+\frac{2}{\beta_2}\mathbb{E}\big[\|\nabla u_{j,l}^n\|^2\big]
+\frac{1}{2\beta_3}\mathbb{E}\big[\|\nabla u_{j,l}^{n-1}\|^2\big]. 
\end{aligned}
\label{e01}
\end{equation}
 The term $\Delta t_l\theta\mathbb{E}\big[\|\nabla u_{j,l}^{n+1}\|^2\big]$ on the left-hand side (LHS) can be
separated into following parts for any $C_1\in (0, 1)$:
\begin{equation}
\begin{aligned}
\Delta t_l\theta\mathbb{E}\big[\|\nabla u_{j,l}^{n+1}\|^2\big]
&=C_1\Delta t_l\theta\mathbb{E}\big[\|\nabla
u_{j,l}^{n+1}\|^2\big]+(1-C_1)\Delta t_l\theta\mathbb{E}\big[\|\nabla
u_{j,l}^{n+1}\|^2-\|\nabla u_{j,l}^n\|^2\big]\\
&+(1-C_1)\Delta t_l\theta\mathbb{E}\big[\|\nabla u_{j,l}^n\|^2\big].
\end{aligned}
\label{e02}
\end{equation}
Substituting (\ref{e00})-(\ref{e02}) into (\ref{S1}), we get
\begin{equation}
\begin{aligned}
&\frac{1}{4}\big(\mathbb{E}\big[\|u_{j,l}^{n+1}\|^2\big]
+\mathbb{E}\big[\|2u_{j,l}^{n+1}-u_{j,l}^n\|^2\big]\big)-\frac{1}{4}\big(\mathbb{E}\big[\|u_{j,l}^n\|^2\big]+\mathbb{E}\big[\|2u_{j,l}^n-u_{j,l}^{n-1}\|^2\big]\big)\\
&+\frac{1}{4}\mathbb{E}\big[\|u_{j,l}^{n+1}-2u_{j,l}^n+u_{j,l}^{n-1}\|^2\big]
+\Big(C_1\theta-\frac{\beta_1}{4}-\frac{\beta_2+\beta_3}{2}\theta_+\Big) \Delta t_l \mathbb{E}\big[\|\nabla
 u_{j,l}^{n+1}\|^2\big]\\
& +(1-C_1)\Delta t_l\theta\mathbb{E}\big[\|\nabla u_{j,l}^{n+1}\|^2-\|\nabla u_{j,l}^n\|^2\big]
+\Big(\frac{2}{3}(1-C_1)\theta-\frac{2\theta_+}{\beta_2}\Big) \Delta t_l \mathbb{E}\big[\|\nabla u_{j,l}^n\|^2\big]\\
&+\Big(\frac{1}{3}(1-C_1)\theta\Big)\Delta t_l\mathbb{E}\big[\|\nabla
u_{j,l}^n\|^2 -\|\nabla
 u_{j,l}^{n-1}\|^2\big]\\
 &+\Big(\frac{1}{3}(1-C_1)\theta-\frac{\theta_+}{2\beta_3}\Big)\Delta t_l\mathbb{E}\big[\|\nabla
 u_{j,l}^{n-1}\|^2\big]
 \leq \frac{\Delta t_l}{\beta_1}\mathbb{E}\big[\|f_j^{n+1}\|_{-1}^2\big].
 \end{aligned}
 \label{e03}
\end{equation}
Selecting $\beta_1=4\delta\theta_+$, $\beta_2=2$, and $\beta_3=1$ for some positive $\delta$, (\ref{e03}) becomes
\begin{equation}
\begin{aligned}
&\frac{1}{4}\mathbb{E}\big[\|u_{j,l}^{n+1}\|^2
+\|2u_{j,l}^{n+1}-u_{j,l}^n\|^2\big]-\frac{1}{4}\mathbb{E}\big[\|u_{j,l}^n\|^2+\|2u_{j,l}^n-u_{j,l}^{n-1}\|^2\big] \\
&+\frac{1}{4}\mathbb{E}\big[\|u_{j,l}^{n+1}-2u_{j,l}^n+u_{j,l}^{n-1}\|^2\big]
+\Big(C_1\theta-\frac{2\delta+3}{2}\theta_+\Big)\Delta t_l\mathbb{E}\big[\|\nabla u_{j,l}^{n+1}\|^2\big]\\
&+(1-C_1)\Delta t_l\theta\mathbb{E}\big[\|\nabla u_{j,l}^{n+1}\|^2-\|\nabla u_{j,l}^n\|^2\big]
 +\Big(\frac{2}{3}(1-C_1)\theta-\theta_+\Big) \Delta t_l \mathbb{E}\big[\|\nabla u_{j,l}^n\|^2\big]\\
&+\Big(\frac{1}{3}(1-C_1)\theta\Big)\Delta t_l\mathbb{E}\big[\|\nabla
u_{j,l}^n\|^2 -\|\nabla
 u_{j,l}^{n-1}\|^2\big]\\
&+\Big(\frac{1}{3}(1-C_1)\theta-\frac{\theta_+}{2}\Big)\Delta t_l\mathbb{E}\big[\|\nabla
 u_{j,l}^{n-1}\|^2\big]
 \leq \frac{\Delta t_l}{4\delta\theta_+}\mathbb{E}\big[\|f_j^{n+1}\|_{-1}^2\big].
 \end{aligned}
  \label{e04}
\end{equation}
Stability follows if the following conditions hold:
\begin{eqnarray}
&&C_1\theta-\frac{2\delta+3}{2}\theta_+\geq 0,\\
%&&\frac{2}{3}(1-C_1)\theta-\theta_+\geq 0,\\
&&\frac{1}{3}(1-C_1)\theta-\frac{\theta_+}{2}\geq 0.
\end{eqnarray}
By taking $C_1=\frac{1}{2}$ and $\delta=\frac{\theta-3\theta_+}{2\theta_+}$, under the assumption (\ref{con1}), we have
$$
C_1\theta-\frac{2\delta+3}{2}\theta_+=\frac{\theta}{2}-\frac{\theta}{2}=0\quad \text{  and  }\quad
\frac{\theta}{3}-\theta_+> 0.
$$
Then, by dropping a positive term, (\ref{e04}) becomes
\begin{equation}
\begin{aligned}
&\frac{1}{4}\mathbb{E}\big[\|u_{j,l}^{n+1}\|^2+\|2u_{j,l}^{n+1}-u_{j,l}^n\|^2\big]
-\frac{1}{4}\mathbb{E}\big[\|u_{j,l}^n\|^2+\|2u_{j,l}^n-u_{j,l}^{n-1}\|^2\big]
\\
&+\frac{\theta}{2}\Delta t_l\mathbb{E}\big[\|\nabla u_{j,l}^{n+1}\|^2-\|\nabla u_{j,l}^n\|^2\big] +\Big(\frac{\theta}{3}-\theta_+\Big)\Delta t_l\mathbb{E}\big[\|\nabla u_{j,l}^n\|^2\big]\\
&+\frac{\theta}{6}\Delta t_l\mathbb{E}\big[\|\nabla u_{j,l}^n\|^2 -\|\nabla u_{j,l}^{n-1}\|^2\big]
%&+\frac{1}{4}\mathbb{E}\big[\|u_{j,l}^{n+1}-2u_{j,l}^n+u_{j,l}^{n-1}\|^2\big]\\
 +\Big(\frac{\theta}{6}-\frac{\theta_+}{2}\Big)\Delta t_l\mathbb{E}\big[\|\nabla u_{j,l}^{n-1}\|^2\big]\\
& \leq \frac{\Delta t_l}{2(\theta-3\theta_+)}\mathbb{E}\big[\|f_j^{n+1}\|_{-1}^2\big].
 \end{aligned}
  \label{S_last}
\end{equation}
Summing (\ref{S_last}) from $n=1$ to $n=N_l-1$ and dropping two positive terms gives
\begin{equation}
\begin{aligned}
&\frac{1}{4}\mathbb{E}\big[\|u_{j,l}^{N_l}\|^2\big]+\frac{1}{4}\mathbb{E}\big[\|2u_{j,l}^{N_l}-u_{j,l}^{N_l-1}\|^2\big]
+\frac{\theta}{2}\Delta t_l\mathbb{E}\big[\|\nabla
u_{j,l}^{N_l}\|^2] \\
%+\frac{\theta}{6}\Delta t_l\mathbb{E}\big[\|\nabla
%u_{j,l}^{N_l-1}\|^2]\\
%+\frac{1}{4}\sum\limits_{n=1}^{N_l-1}\mathbb{E}\big[\|u_{j,l}^{n+1}-2u_{j,l}^n+u_{j,l}^{n-1}\|^2\big]\\
&\qquad+\Big(\frac{\theta}{3}-\theta_+\Big)\Delta t_l\sum\limits_{n=1}^{N_l-1}\mathbb{E}\big[\|\nabla
u_{j,l}^n\|^2\big] \\
%+\Big(\frac{\theta}{6}-\frac{\theta_+}{2}\Big)\Delta t_l\sum\limits_{n=1}^{N_l-1}\mathbb{E}\big[\|\nabla
%u_{j,l}^{n-1}\|^2\big] \\
&\leq 
\frac{\Delta t_l}{2(\theta-3\theta_+)}\sum\limits_{n=1}^{N_l-1}\mathbb{E}\big[\|f_j^{n+1}\|_{-1}^2\big]
+\frac{1}{4}\mathbb{E}\big[\|u_{j,l}^1\|^2\big]+\frac{1}{4}\mathbb{E}\big[\|2u_{j,l}^1-u_{j,l}^0\|^2\big]\\
 & +\frac{\theta}{2}\Delta t_l\mathbb{E}\big[\|\nabla u_{j,l}^1\|^2\big]+\frac{\theta}{6}\Delta t_l\mathbb{E}\big[\|\nabla u_{j,l}^0\|^2\big],
 \end{aligned}
\end{equation}
which completes the proof.
\end{proof}

Then, by using the standard error estimate for the Monte Carlo method (e.g., \cite{liu2013discontinuous}), 
we can bound the statistical error $\mathcal{E}_S^n$ as follows. 

\begin{theorem}\label{th:E_S}
Let $\mathcal{E}_S^n=\mathbb{E}[u_{j,l}^n]-\Psi_{J_l}^n$, where $u_{j,l}^n$ is the result of scheme \eqref{eqn:ens_rand} and $\Psi_{J_l}^n = \frac{1}{J_l}\sum_{j=1}^{J_l}u_{j, l}^n$.
Suppose conditions $\mathsf{(i)}$ and $\mathsf{(ii)}$, and the stability condition (\ref{con1}) hold, there is a generic positive constant $C$ independent of $J_l$ and $\Delta t_l$ such
that
\begin{equation}
\begin{aligned}
&\frac{1}{4}\mathbb{E}\big[\|\mathcal{E}_S^{N_l}\|^2\big]
+\frac{1}{4}\mathbb{E}\big[\|2\mathcal{E}_S^{N_l}-\mathcal{E}_S^{N_l-1}\|^2\big]
%+\frac{1}{4}\sum\limits_{n=1}^{N_l-1}\mathbb{E}\big[\|\mathcal{E}_S^{n+1}-2\mathcal{E}_S^n+\mathcal{E}_S^{n-1}]\|^2\big]\nonumber\\
%&&
+\left(\frac{\theta}{3}-\theta_+\right)\Delta t_l\sum\limits_{n=1}^{N_l-1}\mathbb{E}\big[ \|\nabla\mathcal{E}_S^{n}\|^2 \big]
 \\
&\leq
 \frac{1}{J_l}\Big( \Delta t_l \sum\limits_{n=1}^{N_l}\mathbb{E}\big[\|f_j^{n}\|_{-1}^2]
 + \Delta t_l \mathbb{E}\big[\|\nabla u_{j,l}^1\|^2\big]+ \mathbb{E}\big[\|\nabla u_{j,l}^0\|^2\big]\\
 &+\mathbb{E}\big[\|u_{j,l}^1\|^2\big]+ \mathbb{E}\big[\|2u_{j,l}^1-u_{j,l}^0\|^2\big]\Big).
 \end{aligned}
 \label{th:eq_E_S}
\end{equation}
\par
\end{theorem}
\begin{proof}
% Based on the definition of  $|||\cdot|||$ we have
First, we estimate $\mathbb{E}\big[\|\nabla\mathcal{E}_S^n\|^2\big]$. 
%Define $\langle u_l^n, u_l^n\rangle:=(\nabla u_l^n,\nabla u_l^n)$, then we have
\begin{eqnarray*}
\mathbb{E}\big[ \|\nabla\mathcal{E}_S^n\|^2 \big]&=&\mathbb{E}\Bigg[ \bigg( \frac{1}{J_l}\sum\limits_{i=1}^{J_l} \big( \nabla \mathbb{E}[ u_{i, l}^n ]- \nabla u_{i, l}^n \big ), \frac{1}{J_l}\sum\limits_{j=1}^{J_l} \big( \nabla E[ u_{j, l}^n]- \nabla u_{j, l}^n \big) \bigg) \Bigg]\\
&=&\frac{1}{{J_l}^2}\sum\limits_{i,j=1}^{J_l}\mathbb{E} \Bigg[ \bigg( \nabla \mathbb{E}[ u_l^n]- \nabla u_{i, l}^n, \nabla \mathbb{E}[ u_l^n]-\nabla u_{j, l}^n \bigg) \Bigg]\\
&=&\frac{1}{{J_l}^2}\sum\limits_{j=1}^{J_l}\mathbb{E} \Bigg[ \bigg( \nabla \mathbb{E}[ u_l^n]-\nabla u_{j, l}^n, 
\nabla \mathbb{E}[ u_l^n]- \nabla u_{j, l}^n \bigg) \Bigg].
\end{eqnarray*}
The last equality is due to the fact that $u_{1, l}^n,\ldots, u_{J_l, l}^n$ are i.i.d., and thus the expected value of $\big( \nabla \mathbb{E}[ u_l^n]-\nabla u_{i, l}^n, \nabla \mathbb{E}[ u_l^n]- \nabla u_{j, l}^n \big)$
is a zero for $i\neq j$. We now expand $\mathbb{E}\big[\big( \nabla \mathbb{E}[ u_l^n]-\nabla u_{j, l}^n, 
\nabla \mathbb{E}[ u_l^n]- \nabla u_{j, l}^n \big)\big]$
and use the fact that $\mathbb{E}[\nabla u_{j,l}^n] = \nabla \mathbb{E}[ u_{j,l}^n]$ and $\mathbb{E}[u_l^n]=\mathbb{E}[u_{j,l}^n]$ to obtain
$$
\mathbb{E}\big[ \|\nabla\mathcal{E}_S^n\|^2 \big]=-\frac{1}{J_l}\|\nabla\mathbb{E}[u_{j,l}^n]\|^2+\frac{1}{J_l}\mathbb{E}[\|\nabla u_{j,l}^n\|^2],
$$
which yields
$$
\mathbb{E}\big[ \|\nabla\mathcal{E}_S^n\|^2 \big]\leq
\frac{1}{J_l}\mathbb{E}\big[ \|\nabla u_{j,l}^n\|^2 \big].
$$
With the help pf Theorem \ref{th:stability_R}, we have
\begin{equation}
\begin{aligned}
&\Big(\frac{\theta}{3}-\theta_+\Big)\Delta t_l\sum\limits_{n=1}^{N_l-1}\mathbb{E}\big[\|\nabla \mathcal{E}_S^n\|^2\big]
\leq \frac{1}{J_l}\Big( \frac{\Delta t_l}{\theta-3\theta_+}\sum\limits_{n=1}^{N_l-1}\mathbb{E}\big[\|f_j^{n}\|_{-1}^2] \\
 &\qquad+ \theta\Delta t_l \mathbb{E}\big[\|\nabla u_{j,l}^1\|^2+\|\nabla u_{j,l}^0\|^2\big]
+\mathbb{E}\big[\|u_{j,l}^1\|^2+\|2u_{j,l}^1-u_{j,l}^0\|^2\big]\Big).
 \end{aligned}
\end{equation}
The other terms on the LHS of (\ref{th:eq_E_S}) can be treated in the same manner. This completes the proof.
\end{proof}

Next, we estimate the finite element discretization error $\mathcal{E}_l^n$. \begin{theorem}
\label{th:error} 
Let $\mathcal{E}_l^n=\mathbb{E}[u_j(t_n)-u_{j,l}^n]$, where $u_j(t_n)$ is the solution to equation \eqref{eq:rand} when $\omega=\omega_j$ and $t=t_n$ and $u_{j,l}^n$ is the result of scheme \eqref{eqn:ens_rand}. 
Assume that the initial errors $\|u_j(t_0)-u^0_{j, l}\|$, $\|u_j(t_1)-u^1_{j, l}\|$, $\|\nabla (u_j(t_0)-u^0_{j, l})\|$ and $\|\nabla (u_j(t_1)-u^1_{j, l})\|$ are all at least $\mathcal{O}(h^{m})$. 
Suppose conditions $\mathsf{(i)}$ and $\mathsf{(ii)}$, and the stability condition (\ref{con1}) hold, then there exists a generic constant
$C$ independent of $J_l$, $h_l$ and $\Delta t_l$ such that
\begin{equation}
\begin{aligned}
\frac{1}{4}\mathbb{E}\big[\|\mathcal{E}_l^{N_l}\|^2\big]
+\frac{1}{4}\mathbb{E}\big[\|2\mathcal{E}_l^{N_l} -\mathcal{E}_l^{N_l-1}\|^2\big]
%&&+\frac{1}{4}\sum\limits_{n=1}^{N_l-1}\mathbb{E}\big[\|(u_{j}^{n+1}-u_{j,l}^{n+1})-2(u_j^n-u_{j,l}^n)+(u_j^{n-1}-u_{j,l}^{n-1})\|^2\big]\nonumber\\
%&&+\frac{\theta}{2}\mathbb{E}\big[\|\nabla(u_j^{N_l}-u_{j,l}^{N_l})\|^2\big]+\Big(\frac{\theta}{6}-\frac{\theta_+}{2}\Big)\Delta t_l\mathbb{E}\big[\|\nabla(u_j^{N-1}-u_{j,l}^{N-1})\|^2\big]\nonumber\\
%&+\epsilon\frac{\theta}{2}\Delta t_l\mathbb{E}\big[\|\nabla u_{j}^{N_l}-u_{j,l}^{N_l} \|^2\big]
%+ \epsilon\frac{\theta}{6} \Delta t_l\mathbb{E}\big[\|\nabla\ u_{j}^{N_l-1}-u_{j,l}^{N_l-1}\|^2\big]
&+\Big(\frac{\theta}{3}-\theta_+\Big)\Delta t_l\sum\limits_{n=0}^{N_l}\mathbb{E}\big[\|\nabla \mathcal{E}_l^n \|^2\big] \\
%\\
%&+ \theta\Big(\frac{\epsilon}{3}-\frac{\theta_+}{\theta}\Big)\Delta t_l\sum\limits_{n=1}^{N_l-1}\mathbb{E}\big[\|\nabla u_{j}^{n}-u_{j,l}^{n} \|^2\big]
& \leq C(\Delta t_l^4+h_l^{2m}).
 \end{aligned}
 \label{eq:th_err}
\end{equation}
\end{theorem}
\begin{proof} We first derive the error equation associated to \eqref{eqn:ens_rand}. 
Equation (\ref{eq:rand}) evaluated at $t_{n+1}$ and tested by $\forall\, v_l\in V_l^0$ yields
\begin{equation}
\begin{aligned}
\left( \frac{3u_j(t_{n+1})-4u_j(t_{n})+u_j(t_{n-1})}{2\Delta t_l},v_l \right)
&+(a_j \nabla u_j(t_{n+1}),\nabla v_l)\\
&=(f_j^{n+1},v_l)-(R_j^{n+1},v_l),
\end{aligned}
\label{Eq3}
\end{equation}
where $f_j^{n+1}= f_j(t_{n+1})$ and $R^{n+1}_j=u_{j,t}(t_{n+1})-\frac{3u_j(t_{n+1})-4u_j(t_n)+u_j(t_{n-1})}{2\Delta t_l}$. Denoted by $e_j^n:=u_j(t_n)-u_{j,l}^n$ the
approximation error at the time $t_n$.
Subtracting (\ref{eqn:ens_rand}) from (\ref{Eq3}) produces
\begin{equation}
\begin{aligned}
&\left(\frac{3e_{j}^{n+1}-4e_{j}^n+e_j^{n-1}}{2\Delta t_l},v_l \right)+ (\overline{a}_l\nabla e_{j}^{n+1},\nabla v_l)
+\left( (a_j-\overline{a}_l)\nabla (2e_{j}^n-e_j^{n-1}),\nabla v_l \right)\\
&\qquad+\left( (a_j-\overline{a}_l)\nabla( u_j^{n+1}-2u_j^n+u_j^{n-1}),\nabla v_l \right)+(R^{n+1}_j,v_l)=0. 
\end{aligned}
\label{eq:error}
\end{equation}
Let $P_l(u_j(t_n))$ be the Ritz projection of $u_j(t_n)$ onto $V_l^0$ satisfying 
$$\big(\overline{a}_l \big(\nabla (u_j(t_n)-P_l(u_j(t_n))), \nabla v_l \big)=0,\quad \forall\, v_l\in V_l^0.$$ 
The error can be decomposed as 
$$
e_j^n=\rho_{j,l}^n-\phi_{j,l}^n \text{ with } \rho_{j,l}^n= u_j(t_n)-P_l(u_j(t_n)) \text{ and } \phi_{j,l}^n= u_{j,l}^n-P_l(u_j(t_n)).
$$
By substituting this decomposition into \eqref{eq:error} and choosing $v_l=\phi_{j,l}^{n+1}$, we obtain 
\begin{equation}
\begin{aligned}
&\left( \frac{3\phi_{j,l}^{n+1}-4\phi_{j,l}^n+\phi_{j,l}^{n-1}}{2\Delta t_l}, \phi_{j,l}^{n+1} \right)
+(\overline{a}_l \nabla \phi_{j,l}^{n+1}, \nabla \phi_{j,l}^{n+1})
\\
&\qquad =-\left( (a_j-\overline{a}_l)\nabla (2\phi_{j,l}^n-\phi_{j,l}^{n-1}), \nabla \phi_{j,l}^{n+1} \right)  
+ \left( \frac{3\rho_{j,l}^{n+1}-4\rho_{j,l}^n+\rho_{j,l}^{n-1}}{2\Delta t_l}, \phi_{j,l}^{n+1} \right) \\
&\qquad+(\overline{a}_l\nabla \rho_{j,l}^{n+1}, \nabla \phi_{j,l}^{n+1})
+\left( (a_j-\overline{a}_l)\nabla (2\rho_{j,l}^n-\rho_{j,l}^{n-1}), \nabla \phi_{j,l}^{n+1} \right) \\
&\qquad+\left( (a_j-\overline{a}_l)\nabla(u_{j}^{n+1}-2u_{j}^n+u_j^{n-1}), \nabla \phi_{j,l}^{n+1} \right)
+(R_j^{n+1}, \phi_{j,l}^{n+1}). 
\end{aligned}
\label{eq:error2}
\end{equation}
After integrating over probability space, we have, for the LHS, 
\begin{equation}
\begin{aligned}
\textrm{LHS} &\geq \frac{1}{4\Delta t_l}\mathbb{E}\big[\|\phi_{j,l}^{n+1}\|^2 + \|2\phi_{j,l}^{n+1}-\phi_{j,l}^n\|^2\big] 
-\frac{1}{4\Delta t_l} \mathbb{E}\big[\|\phi_{j,l}^n\|^2+ \|2\phi_{j,l}^n-\phi_{j,l}^{n-1}\|^2 \big] \\
&+ \frac{1}{4\Delta t_l}\mathbb{E}\big[\|\phi_{j,l}^{n+1}-2\phi_{j,l}^n+\phi_{j,l}^{n-1}\|^2\big] +\theta\mathbb{E}\big[\|\nabla\phi_{j,l}^{n+1}\|^2\big].
\end{aligned}
\label{error3}
\end{equation}
We then bound the terms on the RHS of (\ref{eq:error2}) one by one. 
By applying the Cauchy-Schwarz and Young's inequalities, we have 
\begin{equation}
\begin{aligned}
&\mathbb{E}\Big[ \left|\left((a_j-\overline{a}_l )\nabla (2\phi_{j,l}^n-\phi_{j,l}^{n-1}),\nabla \phi_{j,l}^{n+1}\right)\right| \Big] \\
& \leq  \theta_+ \mathbb{E}\big[ |(2\nabla \phi_{j,l}^n,\nabla\phi_{j,l}^{n+1})| \big]
+ \theta_+ \mathbb{E}\big[ |(\nabla\phi_{j,l}^{n-1},\nabla\phi_{j,l}^{n+1})| \big] \\
& \leq \theta_+ \mathbb{E}\big[\|\nabla\phi_{j,l}^{n}\|^2\big]
+ \frac{\theta_+}{2} \mathbb{E}\big[\|\nabla\phi_{j,l}^{n-1}\|^2\big]
+\frac{3 \theta_+}{2} \mathbb{E}\big[\|\nabla\phi_{j,l}^{n+1}\|^2\big].
\end{aligned}
\label{errorfirst}
\end{equation}
%Since,
%\begin{eqnarray*}
%3\rho_{j,l}^{n+1}-4\rho_{j}^{n}+\rho_{j,l}^{n-1}&=&3\rho_{j,l}^n+3\int_{t^n}^{t^{n+1}}\rho_{j,t}(\cdot,\cdot,t)dt-4\rho_{j,l}^n+\rho_{j,l}^n+\int_{t^n}^{t^{n-1}}\rho_{j,t}(\cdot,\cdot,t)dt\nonumber\\
%&=&3\int_{t^n}^{t^{n+1}}\rho_{j,t}(\cdot,\cdot,t)dt-\int_{t^{n-1}}^{t^{n}}\rho_{j,t}(\cdot,\cdot,t)dt,
%\end{eqnarray*}
%\begin{eqnarray*}
%\mathbb{E}\big[\|3\rho_{j,l}^{n+1}-4\rho_{j}^{n}+\rho_{j,l}^{n-1}\|\big]&\leq&3\mathbb{E}\Big[\int_{t^{n-1}}^{t^{n+1}}\|\rho_{j,t}\|dt\Big]
%\leq3\mathbb{E}\Big[\big(\int_{t^{n-1}}^{t^{n+1}}\|\rho_{j,t}\|^2\big)^{1/2}\Big]\big(\int_{t^{n-1}}^{t^{n+1}}1^2\big)^{1/2}\\
%&&=3\sqrt{2\Delta t_l}\mathbb{E}\Big[\big(\int_{t^{n-1}}^{t^{n+1}}\|\rho_{j,t}\|^2\big)^{1/2}\Big]
%\end{eqnarray*}
We further use the Poinc\'{a}re inequality and have 
\begin{eqnarray}
\mathbb{E}\left[ \left|\left(\frac{3\rho_{j,l}^{n+1}-4\rho_{j}^{n}+\rho_{j,l}^{n-1}}{2\Delta t_l},\phi_{j,l}^{n+1}\right)\right| \right]
%&\leq&
%C\mathbb{E}\big[\|\frac{3\rho_{j,l}^{n+1}-4\rho_{j}^{n}+\rho_{j,l}^{n-1}}{2\Delta t_l}\|\big]\mathbb{E}\big[\|\nabla\phi_{j,l}^{n+1}\|\big]\nonumber\\
&\leq&
\frac{C}{4C_0\theta}\mathbb{E}\left[\left \|\frac{3\rho_{j,l}^{n+1}-4\rho_{j}^{n}+\rho_{j,l}^{n-1}}{2\Delta t_l} \right\|^2 \right]
+C_0\theta\mathbb{E}\big[\|\nabla\phi_{j,l}^{n+1}\|^2\big] \nonumber\\
&\leq&
\frac{C}{4C_0\theta}\mathbb{E}\left[ \left\| \frac{1}{\Delta t_l} \int_{t_{n-1}}^{t_{n+1}} \rho_{j, t} \, dt \right\|^2 \right]
+C_0\theta\mathbb{E}\big[\|\nabla\phi_{j,l}^{n+1}\|^2\big] \nonumber\\
&\leq& 
\frac{C}{4C_0\theta \Delta t_l}\mathbb{E}\left[\int_{t^{n-1}}^{t^{n+1}}\|\rho_{j,t}\|^2dt\right]
+C_0\theta\mathbb{E}\big[\|\nabla\phi_{j,l}^{n+1}\|^2\big],
\end{eqnarray}
where $C$ is the Poinc\'{a}re coefficient and $C_0$ is an arbitrary positive constant.  
The rest of terms can be bounded as follows. 
%\begin{eqnarray}
%\left|(\overline{a}_l\nabla\rho_{j,l}^{n+1},\nabla\phi_{j,l}^{n+1})\right|
%\leq
%\frac{ | \overline{a}_l |_{\infty}^2 }{4C_0\theta} \mathbb{E}\big[\|\nabla\rho_j^{n+1}\|^2\big] 
%+C_0 \theta \mathbb{E}\big[\|\nabla\phi_{j,l}^{n+1}\|^2\big].
%\end{eqnarray}
\begin{align}
& \mathbb{E}\left[ \left|(\overline{a}_l\nabla\rho_{j,l}^{n+1},\nabla\phi_{j,l}^{n+1})\right| \right]
= 0. \\
&\mathbb{E}\left[ \left|\left((a_j-\overline{a}_l )\nabla
(2\rho_{j,l}^n-\rho_j^{n-1}),\nabla \phi_{j,l}^{n+1}\right)\right| \right] \\
&\leq \theta_+ \mathbb{E}\big[ |(2\nabla \rho_{j,l}^n,\nabla\phi_{j,l}^{n+1})| \big]
+ \theta_+ \mathbb{E}\big[ |(\nabla\rho_j^{n-1},\nabla\phi_{j,l}^{n+1})| \big] \nonumber \\
&\leq\frac{1}{C_0}\frac{\theta_+^2}{\theta} \mathbb{E}\big[\|\nabla\rho_j^{n}\|^2\big]
+\frac{1}{4C_0}\frac{\theta_+^2}{\theta}\mathbb{E}\big[\|\nabla\rho_j^{n-1}\|^2\big]
+2C_0\theta \mathbb{E}\big[\|\nabla\phi_{j,l}^{n+1}\|^2\big]. \nonumber \\
%
%Since,
%\begin{eqnarray*}
%u_{j}^{n+1}-2u_{j}^{n}+u_{j}^{n-1}&=&u_{j}^{n}+\Delta
%tu_{j,t}^{n}-2u_{j}^{n}+u_{j}^{n}-\Delta
%tu_{j,t}^{n}\\
%&+&\int_{t^n}^{t^{n+1}}(t^{n+1}-t)u_{j,tt}(\cdot,\cdot,t)dt
%+\int_{t^n}^{t^{n-1}}(t^{n-1}-t)u_{j,tt}(\cdot,\cdot,t)dt\\
%&=&\int_{t^{n-1}}^{t^{n+1}}(t^{n+1}-t)u_{j,tt}(\cdot,\cdot,t)dt
%\end{eqnarray*}
%\begin{eqnarray*}
%\mathbb{E}\big[\|\nabla(u_{j}^{n+1}-2u_{j}^{n}+u_{j}^{n-1})\|\big]&=&\mathbb{E}\Big[\int_{t^{n-1}}^{t^{n+1}}\|u_{j,tt}\|(t^{n+1}-t)dt\Big]\\
%&\leq&\mathbb{E}\Big[\big(\int_{t^{n-1}}^{t^{n+1}}\|u_{j,tt}\|^2dt\big)^{1/2}\big(\int_{t^{n-1}}^{t^{n+1}}(t^{n+1}-t)^2dt\big)^{1/2}\Big]\\
%&&=\sqrt{\frac{8\Delta t_l^3}{3}}\mathbb{E}\Big[\big(\int_{t^{n-1}}^{t^{n+1}}\|u_{j,tt}\|^2dt\big)^{1/2}\Big]
%\end{eqnarray*}
&\mathbb{E}\left[ \left|((a_j-\overline{a})\nabla(u_{j}^{n+1}-2u_{j}^n+u_j^{n-1}),\nabla
\phi_{j,l}^{n+1})\right| \right]  \\
&\leq  \frac{1}{4C_0}\frac{\theta_+^2}{\theta} \mathbb{E}\big[\|\nabla(u_{j}^{n+1}-2u_{j}^{n}+u_{j}^{n-1})\|^2\big]
+C_0 \theta \mathbb{E}\big[\|\nabla\phi_{j,l}^{n+1}\|^2\big]\nonumber \\
&\leq \frac{C \Delta t_l^3}{4C_0}\frac{\theta_+^2}{\theta} \mathbb{E}\Big[\int_{t^{n-1}}^{t^{n+1}}\|\nabla u_{j,tt}\|^2dt\Big] 
+C_0 \theta \mathbb{E}\big[\|\nabla\phi_{j,l}^{n+1}\|^2\big], \nonumber
\end{align}
%By the integral form of Taylor's theorem
%\begin{eqnarray*}
%-4u_j^n+u_j^{n-1}&=&-4u_j^{n+1}+4\Delta t_lu_{j,t}^{n+1}-2\Delta t_l^2u_{j,tt}^{n+1}+u_{j}^{n+1}-2\Delta t_lu_{j,t}^{n+1}+2\Delta t_l^2u_{j,tt}^{n+1}\\
%&&\,\,+\frac{1}{2}(4\int_{t^{n+1}}^{t^n}u_{j,ttt}(\cdot,\cdot,t)(t^n-t)^2dt+\int_{t^{n+1}}^{t^{n-1}}u_{j,ttt}(\cdot,\cdot,t)(t^{n-1}-t)^2dt)\\
%&=&-3u_j^{n+1}+2\Delta t_lu_{j,t}^{n+1}+2\int_{t^{n+1}}^{t^n}u_{j,ttt}(\cdot,\cdot,t)(t^n-t)^2dt\\
%&&\,+\frac{1}{2}\int_{t^{n+1}}^{t^{n-1}}u_{j,ttt}(\cdot,\cdot,t)(t^{n-1}-t)^2dt,
%\end{eqnarray*}
%we have
%$$
%\mathbb{E}\big[\|R_{j}^{n+1}\|^2\big]\leq C\Delta t_l^3\mathbb{E}\Big[\int_{t^{n-1}}^{t^{n+1}}\|u_{j,ttt}\|^2dt\Big],
%$$
and
\begin{equation}
\mathbb{E}\left[ \left|(R_j^{n+1},\phi_{j,l}^{n+1})\right| \right] 
\leq
C_0\theta \mathbb{E}\big[\|\nabla\phi_{j,l}^{n+1}\|^2\big]+\frac{C\Delta t_l^3}{C_0\theta}\mathbb{E}\Big[\int_{t^{n-1}}^{t^{n+1}}\|u_{j,ttt}\|^2dt\Big].
\label{error6}
\end{equation}
Substituting (\ref{error3}) to (\ref{error6}) into
(\ref{eq:error2}), we get
\begin{equation}
\begin{aligned}
&\frac{1}{4\Delta t_l}\big(\mathbb{E}\big[\|\phi_{j,l}^{n+1}\|^2\big]
+\mathbb{E}\big[\|2\phi_{j,l}^{n+1}-\phi_{j,l}^n\|^2\big] \big)
-\frac{1}{4\Delta t_l}\big(\mathbb{E}\big[\|\phi_{j,l}^n\|^2\big]+\mathbb{E}\big[\|2\phi_{j,l}^n-\phi_{j,l}^{n-1}\|^2\big]\big) \nonumber\\
&+ \frac{1}{4\Delta t_l}\mathbb{E}\big[\|\phi_{j,l}^{n+1}-2\phi_{j,l}^n+\phi_{j,l}^{n-1}\|^2\big]
+\theta(1 -5C_0-\frac{3\theta_+}{2\theta}) \mathbb{E}\big[\|\nabla\phi_{j,l}^{n+1}\|^2\big] \\
&-\theta_+ \mathbb{E}\big[ \|\nabla \phi_{j,l}^n\|^2 \big]
 - \frac{\theta_+}{2} \mathbb{E}\big[ \|\nabla \phi_{j,l}^{n-1}\|^2 \big] \\
&\leq
\frac{C}{4C_0\theta \Delta t_l}\mathbb{E}\left[\int_{t^{n-1}}^{t^{n+1}}\|\rho_{j,t}\|^2dt\right]
%+\frac{|\overline{a}_l|_{\infty}^2}{4C_0 \theta}\mathbb{E}\big[\|\nabla \rho_{j,l}^{n+1}\|^2\big]
+\frac{\theta_+^2}{C_0\theta}\mathbb{E}\big[\| \nabla\rho_{j}^{n}\|^2\big]
+\frac{\theta_+^2}{4 C_0 \theta}\mathbb{E}\big[\|\nabla\rho_{j,l}^{n-1}\|^2\big]
\nonumber\\
&
+\frac{C\Delta t_l^3}{4 C_0}\frac{\theta_+^2}{\theta}\mathbb{E}\Big[\int_{t^{n-1}}^{t^{n+1}}\|\nabla u_{j,tt}\|^2 dt\Big]
+\frac{C \Delta t_l^3}{C_0 \theta}\mathbb{E}\Big[\int_{t^{n-1}}^{t^{n+1}}\|u_{j,ttt}\|^2 dt\Big].
\end{aligned}
\label{errorlast}
\end{equation}
Now we split the term
$\theta\mathbb{E}\big[\|\nabla\phi_{j,l}^{n+1}\|^2\big]$, and choose $C_0 = \frac{1}{30}(1-\frac{3\theta_+}{\theta})$:
%\begin{eqnarray}
%\theta\mathbb{E}\big[\|\nabla\phi_{j,l}^{n+1}\|^2\big]&=&\varepsilon\theta\mathbb{E}\big[\|\nabla\phi_{j,l}^{n+1}\|^2\big]+(1-\varepsilon)\theta\mathbb{E}\big[\|\nabla\phi_{j,l}^{n}\|^2\big]
%\nonumber\\
%&+&(1-\varepsilon)\theta\mathbb{E}\big[\|\nabla\phi_{j,l}^{n+1}\|^2\big]-\|\nabla\phi_{j,l}^{n}\|^2\big].
%\end{eqnarray}
%Using above equality in (\ref{errorlast}), we get
\begin{equation}
\begin{aligned}
&\frac{1}{4\Delta t_l}(\mathbb{E}\big[\|\phi_{j,l}^{n+1}\|^2\big]
+\mathbb{E}\big[\|2\phi_{j,l}^{n+1}-\phi_{j,l}^n\|^2\big])-\frac{1}{4\Delta
t}(\mathbb{E}\big[\|\phi_{j,l}^n\|^2\big]+\mathbb{E}\big[\|2\phi_{j,l}^n-\phi_{j,l}^{n-1}\|^2\big])
 \\
 &+ \frac{1}{4\Delta t_l}\mathbb{E}\big[\|\phi_{j,l}^{n+1}-2\phi_{j,l}^n+\phi_{j,l}^{n-1}\|^2\big]
 +\frac{\theta}{3} \Big( 1 - \frac{3\theta_+}{\theta} \Big) \mathbb{E}\big[ \|\nabla \phi_{j,l}^{n+1}\|^2 \big] \\
 &+\theta \Big(\frac{1}{3}-\frac{\theta_+}{\theta} \Big) \mathbb{E}\big[ \|\nabla \phi_{j,l}^n\|^2 \big]
 +\theta \Big(\frac{1}{6}-\frac{\theta_+}{2\theta} \Big) \mathbb{E}\big[ \|\nabla \phi_{j,l}^{n-1}\|^2 \big]\\
 &+  \frac{\theta}{2}\Big( \mathbb{E}\big[ \|\nabla \phi_{j,l}^{n+1}\|^2 \big] - \mathbb{E}\big[ \|\nabla \phi_{j,l}^{n}\|^2 \big] \Big)
 +  \frac{\theta}{6}\Big( \mathbb{E}\big[ \|\nabla \phi_{j,l}^{n}\|^2 \big] - \mathbb{E}\big[ \|\nabla \phi_{j,l}^{n-1}\|^2 \big] \Big)  \\
&\leq
\frac{C}{(\theta-3\theta_+)} \bigg\{
\frac{1}{\Delta t_l}\mathbb{E}\left[\int_{t^{n-1}}^{t^{n+1}}\|\rho_{j,t}\|^2dt\right]
%+\frac{|\overline{a}_l|_{\infty}^2}{4C_0 \theta}\mathbb{E}\big[\|\nabla \rho_{j,l}^{n+1}\|^2\big]
+ \theta_+^2 \mathbb{E}\big[\| \nabla\rho_{j}^{n}\|^2\big]
+ \theta_+^2 \mathbb{E}\big[\|\nabla\rho_{j,l}^{n-1}\|^2\big]
\\
&
+{C\Delta t_l^3 \theta_+^2}\mathbb{E}\Big[\int_{t^{n-1}}^{t^{n+1}}\|\nabla u_{j,tt}\|^2 dt\Big]
+{\Delta t_l^3}\mathbb{E}\Big[\int_{t^{n-1}}^{t^{n+1}}\|u_{j,ttt}\|^2 dt\Big] 
\bigg\}.
\end{aligned}
\label{error}
\end{equation}
Summing (\ref{error}) from $n=1$ to $N_l-1$, multiplying both sides by $\Delta t_l$, and dropping several positive terms, we have 
\begin{equation}
\begin{aligned}
&\frac{1}{4}\mathbb{E}\big[\|\phi_{j,l}^{N_l}\|^2\big]+\frac{1}{4}\mathbb{E}\big[\|2\phi_{j,l}^{N_l}-\phi_{j,l}^{N_l-1}\|^2\big]
+ \Big(\frac{\theta}{3}-\theta_+\Big)\Delta t_l\sum\limits_{n=0}^{N_l}\mathbb{E}\big[\|\nabla\phi_{j,l}^{n}\|^2\big]
\\
%+\frac{1}{4}\sum\limits_{n=1}^{N_l-1}\mathbb{E}\big[\|\phi_{j,l}^{n+1}-2\phi_{j,l}^n+\phi_{j,l}^{n-1}\|^2\big]\\
%\\
%&+\theta\frac{\epsilon}{2}\Delta t_l\mathbb{E}\big[\|\nabla\phi_{j,l}^{N_l}\|^2\big]
%+ \theta\frac{\epsilon}{6} \Delta t_l\mathbb{E}\big[\|\nabla\phi_{j,l}^{N_l-1}\|^2\big]
%+ \theta\Big(\frac{\epsilon}{3}-\frac{\theta_+}{\theta}\Big)\Delta t_l\sum\limits_{n=1}^{N_l-1}\mathbb{E}\big[\|\nabla\phi_{j,l}^{n}\|^2\big]
&\leq 
\frac{C}{(\theta-3\theta_+)}
\sum\limits_{n=1}^{N_l-1} \bigg\{
\mathbb{E}\left[\int_{t^{n-1}}^{t^{n+1}}\|\rho_{j,t}\|^2dt\right]
%+\frac{|\overline{a}_l|_{\infty}^2}{4C_0 \theta}\mathbb{E}\big[\|\rho_{j,l}^{n+1}\|^2\big]
+ \Delta t_l \theta_+^2 \mathbb{E}\big[\|\nabla\rho_{j}^{n}\|^2\big]
+ \Delta t_l \theta_+^2 \mathbb{E}\big[\|\nabla\rho_{j,l}^{n-1}\|^2\big]
\\
&
+ \Delta t_l^4 \theta_+^2 \mathbb{E}\Big[\int_{t^{n-1}}^{t^{n+1}}\|\nabla u_{j,tt}\|^2 dt\Big]
+ \Delta t_l^4 \mathbb{E}\Big[\int_{t^{n-1}}^{t^{n+1}}\|u_{j,ttt}\|^2 dt\Big] \bigg\}\\
&+\frac{1}{4}\mathbb{E}\big[\|\phi_{j,l}^{1}\|^2\big]+\frac{1}{4}\mathbb{E}\big[\|2\phi_{j,l}^{1}-\phi_{j,l}^{0}\|^2\big]
+\frac{\theta}{2}\Delta t_l\mathbb{E}\big[\|\nabla\phi_{j,l}^{1}\|^2\big]
+ \frac{\theta}{6} \Delta t_l\mathbb{E}\big[\|\nabla\phi_{j,l}^{0}\|^2\big].
\end{aligned}
\label{error02}
\end{equation}
By the regularity assumption and standard finite element estimates
of Ritz projection error (see, e.g., Lemma 13.1 in \cite{thomee2006galerkin} %Section 4.4 in \cite{brenner2007mathematical}
), namely, for any $u_j^n\in H^{m+1}(D)\cap H_0^1(D)$, 
\begin{equation}
\| \rho_{j,l}^n \|^2\leq Ch_l^{2m+2}\|u_j(t_n)\|_{l+1}^2 
\quad \text{ and }\quad
\|\nabla \rho_{j,l}^n \|^2\leq Ch_l^{2m}\|u_j(t_n)\|_{l+1}^2,
\end{equation}
and use the assumption that $\|e^0_{j, l}\|$, $\|e^1_{j, l}\|$, $\|\nabla e^0_{j, l}\|$, and $\|\nabla e^1_{j, l}\|$ are at least $\mathcal{O}(h^{m})$, we have
\begin{equation}
\begin{aligned}
&\frac{1}{4}\mathbb{E}\big[\|\phi_{j,l}^{N_l}\|^2\big]+\frac{1}{4}\mathbb{E}\big[\|2\phi_{j,l}^{N_l}-\phi_{j,l}^{N_l-1}\|^2\big]
+ \Big(\frac{\theta}{3}-\theta_+\Big)\Delta t_l\sum\limits_{n=0}^{N_l}\mathbb{E}\big[\|\nabla\phi_{j,l}^{n}\|^2\big]
\\
%+\frac{1}{4}\sum\limits_{n=1}^{N_l-1}\mathbb{E}\big[\|\phi_{j,l}^{n+1}-2\phi_{j,l}^n+\phi_{j,l}^{n-1}\|^2\big]\\
%&+\theta\frac{\epsilon}{2}\Delta t_l\mathbb{E}\big[\|\nabla\phi_{j,l}^{N_l}\|^2\big]
%+ \theta\frac{\epsilon}{6} \Delta t_l\mathbb{E}\big[\|\nabla\phi_{j,l}^{N_l-1}\|^2\big]
%+ \theta\Big(\frac{\epsilon}{3}-\frac{\theta_+}{\theta}\Big)\Delta t_l\sum\limits_{n=1}^{N_l-1}\mathbb{E}\big[\|\nabla\phi_{j,l}^{n}\|^2\big]
%\\
&\leq 
\frac{C}{(\theta-3\theta_+)} \Big\{  h_l^{2m+2}
%+\frac{|\overline{a}_l|_{\infty}^2}{4 \theta} h^{2m+2}
+\theta_+^2 h_l^{2m}
+ \Delta t_l^4 \theta_+^2 \mathbb{E}\Big[\int_{0}^{T}\|\nabla u_{j,tt}\|^2 dt\Big]\\
&+  \Delta t_l^4 \mathbb{E}\Big[\int_{0}^{T}\| u_{j,ttt}\|^2 dt\Big]  \Big\}
 + h_l^{2m} + \theta \Delta t_l h_l^{2m},
\end{aligned}
\label{error04}
\end{equation}
where $C$ is a generic constant independent of the time step
$\Delta t_l$ and mesh size $h_l$. By the
triangle inequality, we have 
\begin{equation*}
\begin{aligned}
&\frac{1}{4}\mathbb{E}\big[\|u_j(t_{N_l})-u_{j,l}^{N_l}\|^2\big]
+\frac{1}{4}\mathbb{E}\big[\|2\big(u_j(t_{N_l})-u_{j,l}^{N_l}\big) - \big(u_j(t_{N_l-1})-u_{j,l}^{N_l-1}\big)\|^2\big] \\
&+\Big(\frac{\theta}{3}-\theta_+\Big)\Delta t_l\sum\limits_{n=0}^{N_l}\mathbb{E}\big[\|\nabla \big(u_j(t_{n}-u_{j,l}^{n}\big) \|^2\big] 
\leq C(\Delta t_l^4+h_l^{2m}).
 \end{aligned}
\end{equation*}
Applying Jensen's inequality to terms on the LHS leads to the error estimate (\ref{eq:th_err}).
This completes the proof.
\end{proof}

The combination of the error contributions from the Monte Carlo sampling and finite element
approximation leads to the following estimate for the $l$-th level Monte Carlo ensemble approximation. 
\begin{theorem}
\label{th:Psi}
Let $u(t_n)$ be the solution to equation \eqref{eq:rand} and $\Psi_{J_l}^n = \frac{1}{J_l}\sum_{j=1}^{J_l}u_{j, l}^n$. 
Suppose conditions $\mathsf{(i)}$ and $\mathsf{(ii)}$ hold, and suppose the stability condition (\ref{con1}) is satisfied, then
\begin{equation}
\begin{aligned}
&\frac{1}{4}\mathbb{E}\big[\|\mathbb{E}[u(t_{N_l})]-\Psi_{J_l}^{N_l}\|^2\big]
+\frac{1}{4}\mathbb{E}\big[\|2(\mathbb{E}[u(t_{N_l})]-\Psi_{J_l}^{N_l})
-(\mathbb{E}[u(t_{N_l-1})]-\Psi_{J_l}^{N_l-1})\|^2\big]\\
%&& +\frac{1}{4}\sum\limits_{n=1}^{N_l-1}\mathbb{E}\big[\|(\mathbb{E}[u^{n+1}]-\Psi_{J_l}^{n+1})-2(\mathbb{E}[u^n]-\Psi_l^n)+(\mathbb{E}[u^{n-1}]-\Psi_{J_l}^{n-1}])\|^2\big]\nonumber\\
%&&\quad+\frac{\theta}{2}\mathbb{E}\big[\|\nabla(\mathbb{E}[u^{N_l}]-\Psi_{J_l}^{N_l})\|^2\big]+\Big(\frac{\theta}{6}-\frac{\theta_+}{2}\Big)\Delta
%t\mathbb{E}\big[\|\nabla(\mathbb{E}[u^{N-1}]-\Psi_{J_l}^{N-1})\|^2\big]\nonumber\\
&+\Big( \frac{\theta}{3}-\theta_+ \Big)\Delta t_l
\sum\limits_{n=1}^{N_l}\mathbb{E}[\|\nabla(\mathbb{E}[u(t_n)]-\Psi_{J_l}^{n})\|^2]\\
& \leq \frac{C}{J_l}\Big(\Delta t_l\sum\limits_{n=1}^{N_l}\mathbb{E}\big[\|f_j^{n}\|_{-1}^2]+\Delta t_l
\mathbb{E}\big[\|\nabla u_{j,l}^1\|^2+\|\nabla
 u_{j,l}^0\|^2\big]\\
 &+\mathbb{E}\big[\|u_{j,l}^1\|^2+\|2u_{j,l}^1-u_{j,l}^0\|^2\big]\Big)+C(\Delta t_l^4+h_l^{2m}),
 \end{aligned}
 \label{eq:th_error}
\end{equation}
 where $C$ is a positive constant independent of $J_l,\Delta t_l$ and $h_l$.\\
\begin{proof} Consider the first term on the LHS of
(\ref{eq:th_error}). By the triangle and Young's inequality, we
get
$$
\mathbb{E}\big[\|\mathbb{E}[u(t_{N_l})]-\Psi_{J_l}^{N_l}\|^2\big]
\leq
2\big(\mathbb{E}\big[\|\mathbb{E}[u_j(t_{N_l})]-\mathbb{E}[u_{j,l}^{N_l}]\|^2\big]+\mathbb{E}\big[\|\mathbb{E}[u_{j,l}^{N_l}]-\Psi_{J_l}^{N_l}\|^2\big]\big).
$$
%Applying Jensen's inequality to the first term on the RHS of the
%above inequality, we have
%$$
%\mathbb{E}\big[\|\mathbb{E}[u_j(t_{N_l})]-\mathbb{E}[u_{j,l}^{N_l}]\|^2\big]
%\leq
%\mathbb{E}\big[\mathbb{E}[\|u_j(t_{N_l})-u_{j,l}^{N_l}\|^2]\big]
%=\mathbb{E}\big[\|u_j(t_{N_l})-u_{j,l}^{N_l}\|^2\big].
%$$
Then the conclusion follows from Theorems \ref{th:E_S}-\ref{th:error}. The other terms on the LHS of
(\ref{eq:th_error}) can be estimated in the same manner.
\end{proof}
\end{theorem}\par
%It is seen from \eqref{eq:th_error} that the spatial and temporal discretization errors are equilibrated if $\mathcal{O}\left(\Delta t_l^4/h_l^{2m}\right)= 1$. Moreover, the optimal choice of sample size is $J_l= \mathcal{O}(h_l^{-2m})$, which is on the order of $2^{2ml} J_0$.

\subsection{Ensemble-based multi-level Monte Carlo finite element method}\label{sec:analysis_ml}

Now, we derive the error estimate for the EMLMC method. 
\begin{theorem}
\label{th:E_L}
Suppose conditions $\mathsf{(i)}$ and $\mathsf{(ii)}$ and the stability condition (\ref{con1}) hold, then
the EMLMC approximation error satisfies 
\begin{equation}
\begin{aligned}
&\frac{1}{4}\mathbb{E}\left[\left\|\mathbb{E}\big[u(t_{N_L})\big]- \Psi\big[u_L(t_{N_L})\big]\right\|^2\right]
+\frac{1}{4}\mathbb{E}\Big[\big\|\mathbb{E}\big[u^{N_L}]-\Psi[u_L(t_{N_L})\big] - \big(\mathbb{E}\big[ u^{N_{L}-1} \big] \\ 
&-\Psi \big[ u_L(t_{N_{L}-1}) \big] \big)\big\|^2\Big]
+\Big(\frac{\theta}{3}-{\theta_+}\Big)\Delta t_L\sum\limits_{n=1}^{N_L}\mathbb{E}\Big[ \big\| \nabla \mathbb{E} \big[ u(t_{n}) \big]
- \nabla \Psi \big[u_L(t_{n}) \big] \big\|^2 \Big]\\
&\leq C\Big(h_L^{2m}+\Delta t_L^4+\sum\limits_{l=1}^L\frac{1}{J_l}(h_l^{2m}+\Delta t_l^4)\Big)+\frac{C}{J_0}\Big(\Delta t_0\sum\limits_{n=1}^{N_0}\mathbb{E}\big[\|f_j^{n}\|_{-1}^2]\\
 &+\Delta t_0
\mathbb{E}\big[\|\nabla u_{j,0}^1\|^2+\|\nabla
 u_{j,0}^0\|^2\big]+\mathbb{E}\big[\|u_{j,0}^1\|^2+\|2u_{j,0}^1-u_{j,0}^0\|^2\big]\Big),
 \end{aligned}
\label{th:eq_E_L}
\end{equation}
\end{theorem}
 where $C>0$ is a constant independent of $J_l,\Delta t_l$ and $h_l$.\\
\begin{proof}
We only analyze the first term on the LHS because the other terms can be treated in the same manner. 
First, we introduce $u_{-1}(t)=0$. 
\begin{equation}
\begin{aligned}
&\mathbb{E}\left[\big\|\mathbb{E}[u(t_{N_L})]-\Psi[u_L(t_{N_L})] \big\|^2\right]\\
&=\mathbb{E}\Big[\big\|\mathbb{E}[u(t_{N_L})]-\mathbb{E}[u_L(t_{N_L})]+\mathbb{E}[u_L(t_{N_L})]
-\sum\limits_{l=0}^L\Psi_{J_l}[u_l(t_{N_L})-u_{l-1}(t_{N_L})]\big\|^2\Big]\\
&\leq C
\bigg(\mathbb{E}\Big[\big\|\mathbb{E}[u(t_{N_L})]-\mathbb{E}[u_L(t_{N_L})]\big\|^2\Big]
+\sum\limits_{l=0}^L\mathbb{E}\Big[\big\|\Big(\mathbb{E}[u_l(t_{N_L})-u_{l-1}(t_{N_L})] \\
&\qquad-\Psi_{J_l}[u_l(t_{N_L})-u_{l-1}(t_{N_L})]\Big)\big\|^2\Big]\bigg).
\end{aligned}
\label{eq:s1}
\end{equation}
%We calculate the error bounds for the term $I$ and $II$ separately.

By Jensen's inequality and Theorem \ref{th:error}, we get
\begin{equation}
\begin{aligned}
\mathbb{E}\Big[\big\|\mathbb{E}[u(t_{N_L})]-\mathbb{E}[u_L(t_{N_L})]\big\|^2\Big] 
&\leq \mathbb{E}\Big[\big\|u(t_{N_L})-u_L(t_{N_L})\big\|^2\Big] \\
&\leq C(\Delta t_L^4+h_L^{2m}).
\end{aligned}
\label{eq:exp}
\end{equation}

By Theorems \ref{th:E_S}-\ref{th:error} and the triangle inequality, we have 
\begin{equation}
\begin{aligned}
&\mathbb{E}\Big[\big\|\mathbb{E}[u_l(t_{N_L})-u_{l-1}(t_{N_L})]-\Psi_{J_l}[u_l(t_{N_L})-u_{l-1}(t_{N_L})]\big\|^2\Big]\\
&=\mathbb{E}\Big[\big\|(\mathbb{E}-\Psi_{J_l})[u_l(t_{N_L})-u_{l-1}(t_{N_L})]\big\|^2\Big] \\
&\leq \frac{1}{J_l}\mathbb{E}\big[\|u_l(t_{N_L})-u_{l-1}(t_{N_L})\|^2\big]\\
&\leq \frac{2}{J_l}\Big(\mathbb{E}\big[\|u(t_{N_L})-u_l(t_{N_L})\|^2\big]
+\mathbb{E}\big[\|u(t_{N_L})-u_{l-1}(t_{N_L})\|^2\big]\Big)\\
&\leq \frac{C}{J_l}\big(\Delta t_l^4+h_l^{2m}+\Delta t_{l-1}^4+h_{l-1}^{2m}\big)
\leq \frac{C}{J_l}\big(\Delta t_l^4+h_l^{2m}\big).
\end{aligned}
\label{J}
\end{equation}

Meanwhile, based on Theorem \ref{th:Psi}, we have
\begin{equation}
\begin{aligned}
&\mathbb{E}\big[\|\mathbb{E}[u_0(t_{N_L})]-\Psi_{J_0}[u_0(t_{N_L})]\|^2\big] \\
&\leq \frac{C}{J_0}\Big(\Delta t_0\sum\limits_{n=1}^{N_0}\mathbb{E}\big[\|f_j^{n+1}\|_{-1}^2\big]
+\Delta t_0\mathbb{E}\big[\|\nabla u_{j,0}^1\|+\|\nabla u_{j,0}^0\|^2\big] \\
&+\mathbb{E}\big[\|u_{j,0}^1\|^2+\|2u_{j,0}^1-u_{j,0}^0\|^2\big]\Big).
\end{aligned}
\label{eq:J0}
\end{equation}

Plugging \eqref{eq:exp}, \eqref{J} and \eqref{eq:J0} into \eqref{eq:s1}, we have
\begin{equation}
\begin{aligned}
&\frac{1}{4}\mathbb{E}\big[\|\mathbb{E}[u(t_{N_L})]-\Psi[u_L(t_{N_L})]\|^2\big]
\leq C\Big(\Delta t_L^4 + h_L^{2m}+\sum\limits_{l=1}^L\frac{1}{J_l} (\Delta t_l^4 + h_l^{2m})\Big)\\
&+\frac{C}{J_0}\Big(\Delta t_0 \sum\limits_{n=1}^{N_0}\mathbb{E}\big[\|f_j^{n}\|_{-1}^2] +\Delta t_0
\mathbb{E}\big[\|\nabla u_{j,0}^1\|^2+\|\nabla u_{j,0}^0\|^2\big]\\
&+\mathbb{E}\big[\|u_{j,0}^1\|^2+\|2u_{j,0}^1-u_{j,0}^0\|^2\big]\Big).
 \end{aligned}
\end{equation}
The other terms on the LHS of (\ref{th:eq_E_L}) can be treated in the same manner. This completes the proof.
\end{proof}

Since, in general, the finite element simulation cost increases as the mesh is refined, we can balance the time step size $\Delta t_l$, mesh size $h_l$ and sampling size $J_l$ in the preceding error estimation for achieving an optimal rate of convergence while keeping the computational cost at minimum. 
\begin{corollary}
By taking 
$$\Delta t_l= \mathcal{O}(\sqrt{h_l^m}) \quad\text{ and }\quad J_l=  l^{1+\varepsilon}2^{2m(L-l)} J_L$$ 
for an arbitrarily small positive constant $\epsilon$ and $l=0, 1, \cdots, L$, the EMLMC approximation satisfies 
\begin{equation}
\begin{aligned}
&\frac{1}{4}\mathbb{E}\left[\left\|\mathbb{E}\big[u(t_{N_L})\big]- \Psi\big[u_L(t_{N_L})\big]\right\|^2\right]
+\frac{1}{4}\mathbb{E}\Big[\big\|\mathbb{E}\big[u^{N_L}]-\Psi[u_L(t_{N_L})\big] - \big(\mathbb{E}\big[ u^{N_{L}-1} \big] \\ 
&-\Psi \big[ u_L(t_{N_{L}-1}) \big] \big)\big\|^2\Big]
+\Big(\frac{\theta}{3}- \theta_+ \Big)\Delta t_L\sum\limits_{n=1}^{N_L}\mathbb{E}\Big[ \big\| \nabla \mathbb{E} \big[ u(t_{n}) \big]
- \nabla \Psi \big[u_L(t_{n}) \big] \big\|^2 \Big]\\
&\leq Ch_L^{2m},
\end{aligned}
\end{equation}
where $C>0$ are constants independent of $J_l,\Delta t_l$ and $h_l$.
\label{eq:error_analysis}
\end{corollary}

%\begin{proof}
%The convergence result in Theorem (\ref{th:E_L}) suggests that we
%choose $J_l$ such that the overall rate of convergence is
%$O(h_L^{2m})$. With the choice
%\begin{equation}
%J_l=O\big(l^{1+\varepsilon}(h_l/h_L)^{2m}\big)=O\big(l^{1+\varepsilon}
%2^{2m(L-l)}\big),\quad \quad l=1,2,\cdots, L
%\end{equation}
%for some $\epsilon>0$, we obtain from (\ref{th:eq_E_L}) the
%asserted error bound, since for $\epsilon>0$ this implies
%\begin{eqnarray}
%\sum\limits_{l=1}^Lh_l^{2m}J_l^{-1}&\leq& C\sum\limits_{l=1}^L
%(2^{-l}h_0)^{2m}l^{-(1+\varepsilon)}2^{2m(l-L)}
%=C\sum\limits_{l=1}^L2^{-2mL}h_0^{2m}l^{-(1+\varepsilon)}\nonumber\\
%&= & C(2^{-L}h_0)^{2m}\sum\limits_{l=1}^L
%l^{-(1+\varepsilon)}=C(\varepsilon)h_L^{2m},
%\end{eqnarray}
%%\begin{eqnarray}
%%\sum\limits_{l=1}^L\Delta t_l^4J_l^{-1}&\leq&
%%C\sum\limits_{l=1}^Ll^{-(1+\varepsilon)}2^{2m(l-L)}\Delta t_l^4=
%%C\sum\limits_{l=1}^Ll^{-(1+\varepsilon)}(\frac{h_L}{h_l})^{2m}\Delta
%%t^4\nonumber\\
%%&\leq& C\Delta t_l^4\sum\limits_{l=1}^L
%%l^{-(1+\varepsilon)}=C(\varepsilon)\Delta t_l^4,
%%\end{eqnarray}
%and
%\begin{eqnarray}
%\frac{1}{J_0}\leq C(2^{-Lh_0})^{2m}=Ch_L^{2m}.
%\end{eqnarray}
%This completes the proof.
%\end{proof}

\section{Numerical Experiments}\label{sec:num}
In this section, we apply the proposed ensemble-based multilevel Monte Carlo algorithm to two numerical tests 
for solving the random parabolic equation \eqref{eq:rand}. The goal is two-fold: to illustrate the theoretical results in Test 1; and to show the efficiency of the proposed method in Test 2. 

\subsection{Test 1}
In this experiment, we check the convergence rate of the EMLMC method numerically by considering a problem with an {\em a priori} known exact solution. 
The diffusion coefficient and the exact solution are selected as follows.
\begin{equation*}
\begin{aligned}
a(\omega, \bx)&= 8+(1+\omega)\sin(xy), \\
u(\omega, \bx, t)&= (1+\omega)[\sin(2\pi x)\sin(2\pi y)+\sin(4\pi t)],
\end{aligned}
\end{equation*}
where $\omega$ obeys a uniform distribution on $[-\sqrt{3}, \sqrt{3}]$, $t\in [0, 1]$, and $(x, y)\in [0, 1]^2$.
The initial condition, inhomogeneous Dirichlet boundary condition and source term are chosen to match the prescribed exact solution.
Therefore, the expectation of the solution is 
\begin{equation*}
\mathbb{E}[u]= \sin(2\pi x)\sin(2\pi y)+\sin(4\pi t).
\end{equation*}

For the spatial discretization, we use quadratic finite elements on uniform triangulations, that is, $m=2$. 
To verify the analysis given in \eqref{eq:error_analysis}, we fix $L$ and choose the mesh size $h_l =  \sqrt{2} \cdot 2^{-2-l}$, time step size $\Delta t_l =  2^{-3-l}$, and number of samples $J_l= 2^{4(L-l)+1}$ at the $l$-th level, where $l= 0, \ldots, L$ in the EMLMC simulation. 
The experiment is repeated for $R=10$ times. Let 
$$
\mathcal{E}_{L^2}=  \sqrt{\frac{1}{R} \sum_{r=1}^{R}\left\|\mathbb{E}\big[u(T)\big]- \Psi\big[u_L^{(r)}(t_{N_L})\big]\right\|^2 } \,,
$$
$$
\mathcal{E}_{H^1}= \sqrt{\frac{1}{RM} \sum_{r=1}^{R} \sum_{m=1}^M \left\|\mathbb{E}\big[\nabla u(t_m)\big]- \Psi\big[\nabla u_L^{(r)}(t_m)\big]\right\|^2 }\,,
$$
where $u$ is the exact solution and $u_L^{(r)}$ is the EMLMC solution of the $r$-th replica. 
Hence, $\mathcal{E}_{L^2}$ and $\mathcal{E}_{H^1}$ represent the numerical error in $L^2$ and $H^1$ norms, respectively. 
With the above choice of discretization and sampling strategy, we expect both quantities converge quadratically with respect to $h_L$ as derived in Corollary \ref{eq:error_analysis} .   

\begin{table}[htp]
\centering
{\footnotesize
\caption{Numerical errors of the EMLMC. }
\label{tab:t2u}%
\begin{tabular}{|c|c|c|c|c|}
\hline
$L$& $\mathcal{E}_{L^2}$ & rate & $\mathcal{E}_{H^1}$ & rate  \\
\hline
1 & $6.11\times 10^{-2}$ &   -  & $5.60\times 10^{-1}$ &   -  \\
2 & $1.43\times 10^{-2}$ &  2.10  & $1.50\times 10^{-1}$ &  1.90   \\
3 & $3.60 \times 10^{-3}$ &  1.99   & $3.81\times 10^{-2}$ & 1.98    \\
\hline
\end{tabular}
}
\end{table}

The EMLMC numerical errors as $L$ varies from 1 to 3 are listed in Table \ref{tab:t2u}. It is observed that both $\mathcal{E}_{L^2}$ and $\mathcal{E}_{H^1}$ converge at the order of nearly 2 with respect to $h_L$, which matches our expectation.

\subsection{Test 2}
%-----------------------------------------------------
Next, we use a test problem to demonstrate the effectiveness of the EMLMC method. 
The same test problem was considered in \cite{luo2017ensemble} for testing the first-order, ensemble-based Monte Carlo method and a similar computational setting was used in \cite{nobile2009analysis} to compare numerical approaches for parabolic equations with random coefficients.

The test problem is associated with the zero forcing term $f$, zero initial conditions,
and homogeneous Dirichlet boundary conditions on the top, bottom
and right edges of the domain but inhomogeneous Dirichlet
boundary condition, $u= y(1-y)$, on the left edge. The random
coefficient varies in the vertical direction and has the following
form
\begin{equation}
a(\omega, \bx) = a_0 + \sigma \sqrt{\lambda_0} Y_0(\omega) +
\sum_{i=1}^{n_f} \sigma \sqrt{\lambda_i}\left[Y_i(\omega)\cos(i\pi
y) + Y_{n_f+i}(\omega)\sin(i\pi y)\right] \label{eq:a}
\end{equation}
with $\lambda_0 = \frac{\sqrt{\pi} L_c}{2}$, $\lambda_i =
\sqrt{\pi} L_c e^{-\frac{(i \pi L_c)^2}{4}}$ for $i=1, \ldots,
n_f$ and $Y_0$, \ldots, $Y_{2n_f}$ are uncorrelated random
variables with zero mean and unit variance. In the following numerical test,
we take $a_0= 1$, $L_c = 0.25$, $\sigma = 0.15$, $n_f = 3$ and
assume the random variables $Y_0, \ldots, Y_{2n_f}$ are
independent and uniformly distributed in the interval $[-\sqrt{3},
\sqrt{3}]$.
We use quadratic finite elements for spatial discretization and simulate the system over the
time interval $[0, 0.5]$.
%This choice of final time guarantees a steady-state can be achieved at the end of simulations.

We use the EMLMC method to analyze some stochastic information of the system such as the expectation of the solution at final time.
More precisely, we apply the EMLMC with the maximum level $L= 2$, the mesh size $h_l =  \sqrt{2} \cdot 2^{-3-l}$, time step size $\Delta t_l =  2^{-4-l}$, and number of samples $J_l= 2^{4(L-l)+1}$ at the $l$-th level, for $l= 0, \ldots, L$. 
Note that if the samples does not satisfy the stability condition \eqref{con1}, we will divide the sample set into small subsets so that \eqref{con1} holds on each smaller group. Since the diffusion coefficient function is independent of time, such a process can be efficiently implemented for ensemble calculations at each level. 
The EMLMC solution at the final time $T$ is 
$$\Psi_h^E(\bx) = \Psi[u_L^E(t_{N_L})], $$ 
which is shown in Figure \ref{fig:comp} (left). 
Note that due to the small size of the problem, we apply LU factorization in solving the linear systems. 
%The total number of points involved in the simulations is 546. 
Since the exact solution is unknown, to quantify the performance of the EMLMC method, we compare the result with that of the MLMC finite element simulations using the same computational setting. 
The same set of sample values is used, thus, the only difference is $J_l$ individual finite element simulations are implemented at $l$-th level.

\begin{figure}[htp]
\centering
\begin{minipage}{0.32\textwidth}
\includegraphics[width=.9\textwidth]{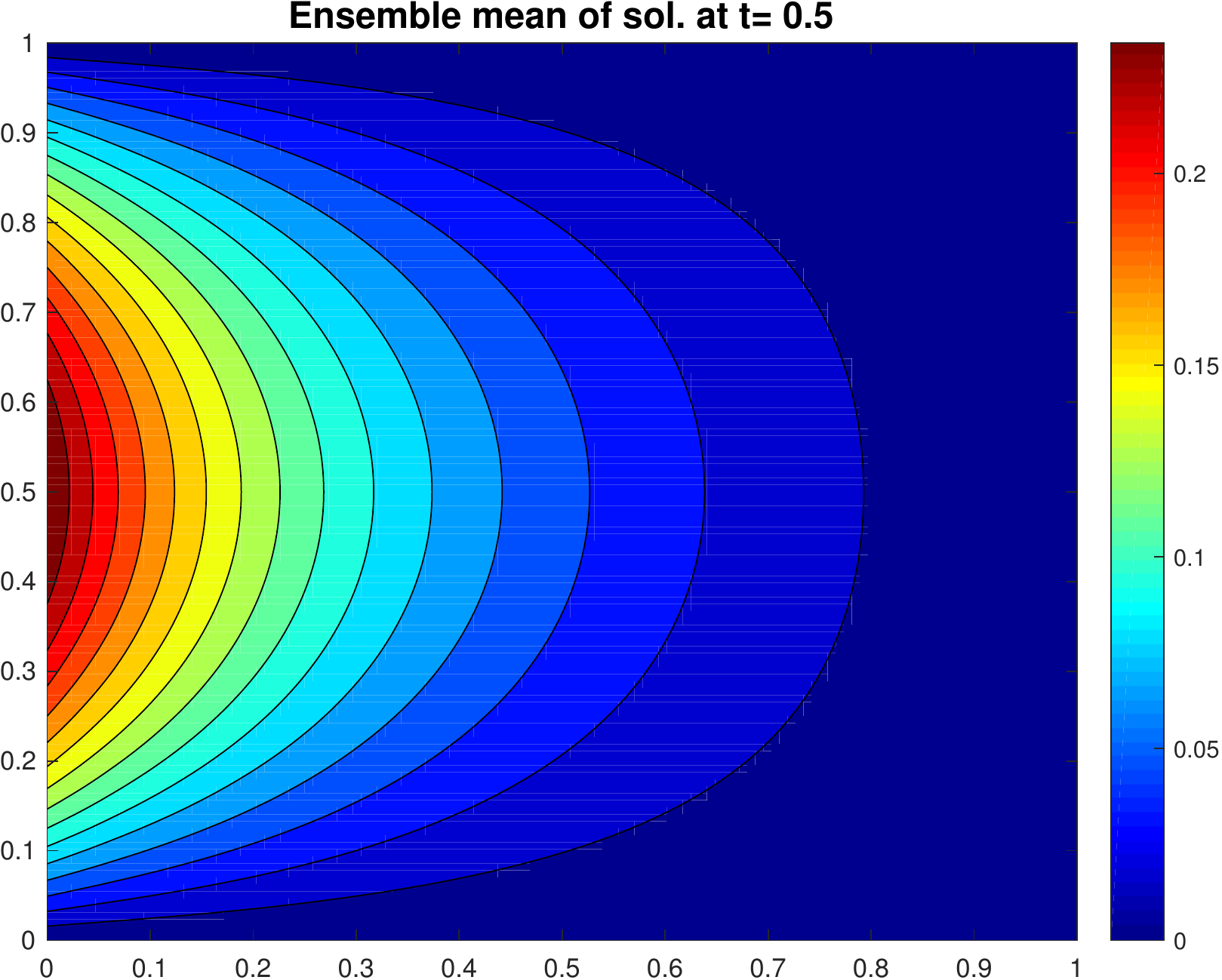}
\end{minipage}
\begin{minipage}{0.32\textwidth}
\includegraphics[width=.9\textwidth]{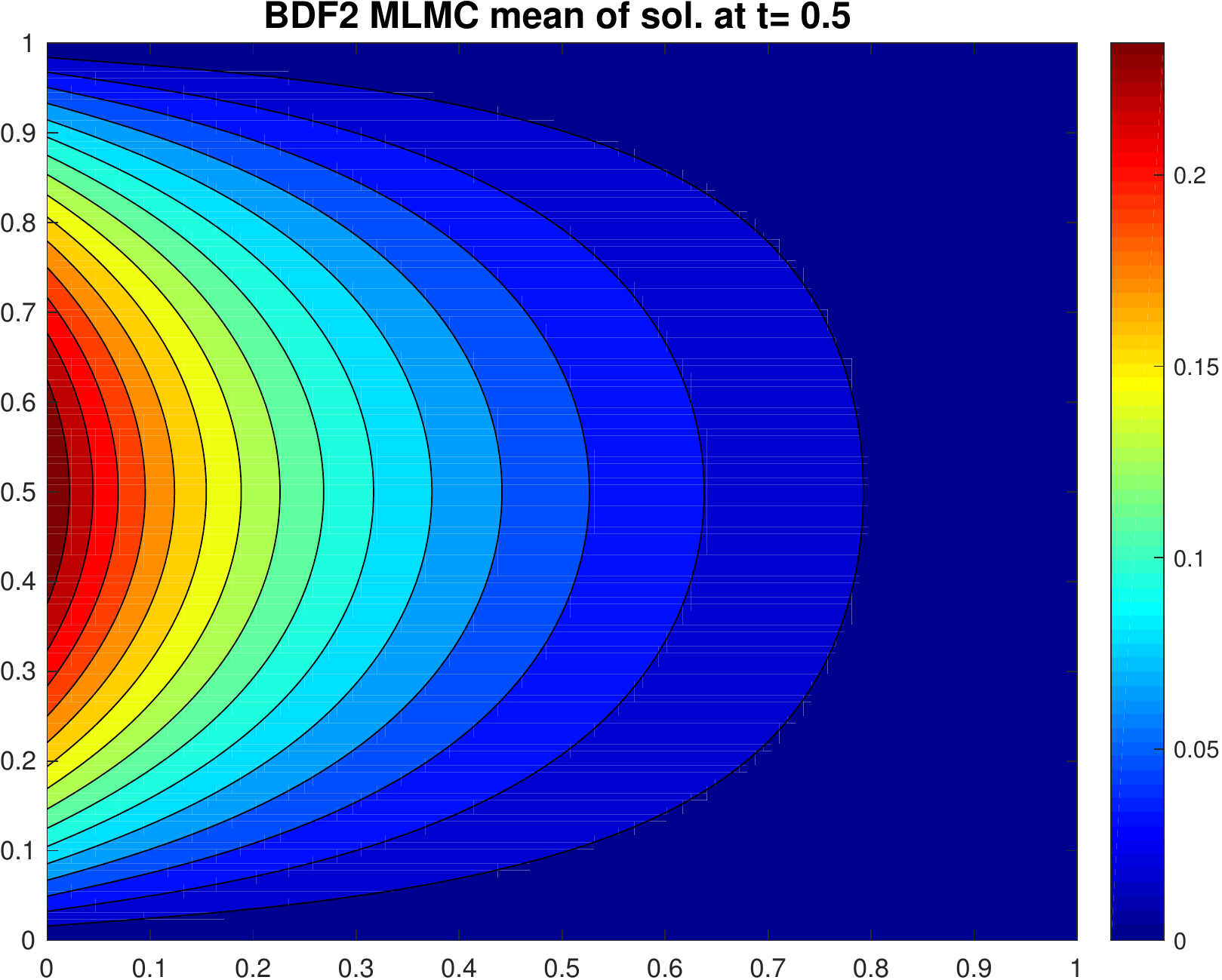}
\end{minipage}
\begin{minipage}{0.32\textwidth}
\includegraphics[width=.9\textwidth]{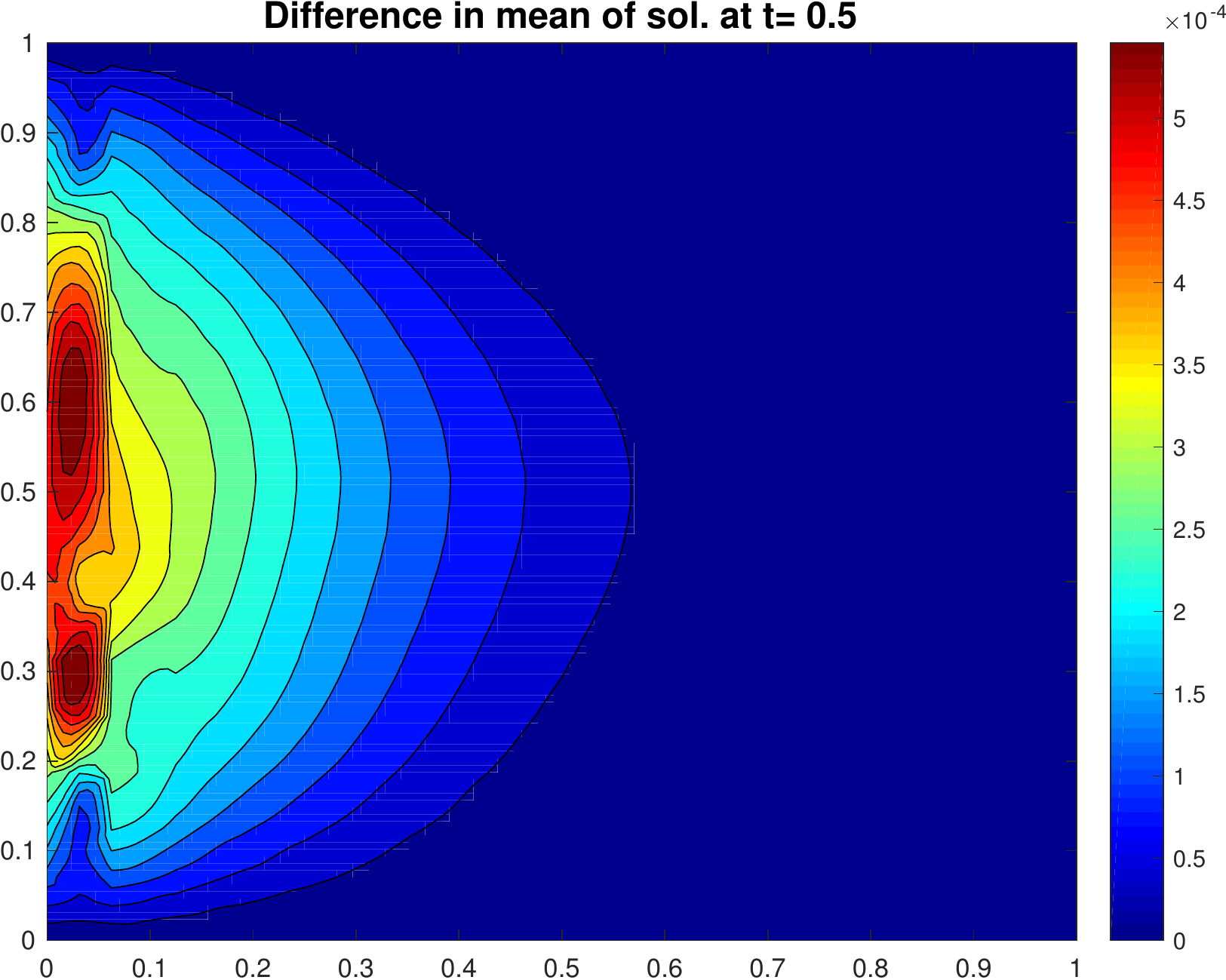}
\end{minipage}
\caption{Comparison of the simulation mean: ensemble simulations (left), finite element simulations (middle), and the associated difference (right).} \label{fig:rand_mean_std}
\label{fig:comp}
\end{figure}

Denote the approximated expected value of the latter approach by
 $$\Psi_h^I(\bx) = \Psi[u_L^I(t_{N_L})],$$
which is shown in Figure \ref{fig:comp} (middle). 
Note that for a fair comparison, we also use the LU factorization in solving all the linear systems in individual simulations. 
The difference between $\Psi_h^E$ and $\Psi_h^I$, $|\Psi_h^E-\Psi_h^I|$, is shown in Figure \ref{fig:comp} (right). 
It is observed that the difference is on the order of $10^{-4}$, which indicates the EMC method is able to provide the same accurate approximation as individual simulations.
However, the CPU time for the ensemble simulation is $2.65\times 10^3$ seconds, while that of the individual simulations is $1.01\times 10^4$ seconds.

\section{Conclusions}\label{sec:con}

A multilevel Monte Carlo ensemble method is developed in this paper to solve second-order random parabolic partial differential equations. This method naturally combines the ensemble-based, multilevel Monte Carlo sampling approach with a second-order, ensemble-based time stepping scheme so that the computational efficiency for seeking stochastic solutions is improved. Numerical analysis shows the numerical approximation achieve the optimal order of convergence.  As a next step, we will extend the method to large-scale, nonlinear partial differential equations, in which we will deal with nonlinearity of the system and use iterative block solver to solve high-dimensional linear systems efficiently.  

\bibliographystyle{siamplain}

\bibliography{ensemble_MLMC}
\end{document}